\documentclass[12pt,regno]{amsart}
\usepackage{ifthen,algorithm,algorithmic}
\usepackage{amssymb,amsfonts,amsmath,latexsym,dsfont,amsthm}
\usepackage{tikz,pgflibraryplotmarks}
\usepackage{fullpage}
\usepackage{graphicx}
\usepackage{mathrsfs}
\usepackage{setspace}
\usepackage[english]{babel}

\graphicspath{{images/}{./}}

\usepackage{boxedminipage}

\usepackage{pgf} 
\usepackage{subfigure} 

\newcommand{\FrameboxA}[2][]{#2}
\newcommand{\Framebox}[1][]{\FrameboxA}

\renewcommand{\pmod}[1]{{\ifmmode\text{\rm\ (mod~$#1$)}\else\discretionary{}{}{\hbox{ }}\rm(mod~$#1$)\fi}}
\newcommand{\p}{\mathcal{P}}
\newcommand{\LL}{\mathcal{L}}

\newcommand{\phik}{\phi_{k}(n)}
\newcommand{\lamk}{\lambda_{k}(n)}
\newcommand{\hk}{h_{k}}
\newcommand{\abs}[1]{\lvert#1\rvert}
\newcommand{\nat}{\mathbb{N}}
\newcommand{\Z}{\mathbb{Z}}
\newcommand{\al}{\alpha}
\newcommand{\eps}{\epsilon}
\newcommand{\lcm}{\mathop{\rm lcm }}
\newcommand{\li}{\mathop{\rm li }}

\newtheorem{theorem}{Theorem}
\newtheorem{lemma}[theorem]{Lemma}
\newtheorem{proposition}[theorem]{Proposition}

\theoremstyle{definition}
\newtheorem{definition}[theorem]{Definition}

\title{The iterated Carmichael lambda function}
\author{Nick Harland}
\address{Department of Mathematics, University of British Columbia, Room 121, 1984 Mathematics Road, Vancouver, BC, V6T 1Z2, Canada}
\email{harlandn@math.ubc.ca}

\begin{document}
\maketitle
\begin{abstract}
The Carmichael lambda function $\lambda(n)$ is defined to be the smallest positive integer $m$ such that $a^m$ is congruent to $1$ modulo $n,$ for all $a$ and $n$ relatively prime. The function $\lambda_k(n)$ is defined to be the $k$th iterate of $\lambda(n).$ Previous results show a normal order for $n/\lambda_k(n)$ where $k=1,2.$ We will show a normal order for all $k.$
\end{abstract}

\section{Introduction}\label{Intro}

The Carmichael lambda function $\lambda(n)$ is defined to be the order of the largest cyclic subgroup of the multiplicative subgroup $(\Z / n\Z)^{\times}.$ It can be computed using the identity $\lambda(\lcm\{a,b\})=\lcm\{\lambda(a),\lambda(b)\}$ and its values at prime powers which are $\lambda(p^{k})=\phi(p^{k})=p^k-p^{k-1}$ for odd primes $p$ and $\lambda(2)=1,\lambda(4)=2,$ and $\lambda(2^{k})=\phi(2^{k})/2=2^{k-2}$ for $k \ge 3$.

Several properties of $\lambda(n)$ were studied by Erd\H os, Pomerance, and Schmutz in \cite{EPS}. In particular they showed that $\lambda(n)=n \exp(-(1+o(1))\log\log n \log\log\log n)$ as $n \rightarrow \infty$ for almost all $n$. Martin and Pomerance showed in \cite{MP} that $\lambda(\lambda(n))=n \exp(-(1+o(1))(\log\log n)^2 \log\log\log n)$ as $n \rightarrow \infty$ for almost all $n$. The $k$--fold iterated Carmichael lambda function is defined recursively to be 
$$\lambda_1(n)=\lambda(n),~~\lambda_k(n)=\lambda(\lambda_{k-1}(n)).$$ We define $\phi_k(n)$ similarly.
In \cite{MP} it is conjectured that $$\lambda_{k}(n)=n \exp\bigg({-}\frac{1}{(k-1)!}(1+o_k(1))(\log\log n)^k \log\log\log n \bigg)$$ for almost all $n.$ In this paper we prove that conjecture.

\begin{theorem} \label{MainTheorem} For fixed $k$, the normal order of $\log \frac{n}{\lambda_{k}(n)}$ is $\frac{1}{(k-1)!}(\log\log{n})^{k}\log\log\log{n}$. \end{theorem}

We'll actually prove the theorem in the following slightly stronger form. Given any function $\psi(x)=o(\log\log\log x)$ and $\psi(x) \rightarrow \infty$ as $x \rightarrow \infty$ we have
$$\log\bigg(\frac{n}{\lambda_{k}(n)}\bigg) = \frac{1}{(k-1)!} (\log\log n)^{k} \bigg(\log\log\log n + O_k\big(\psi(n)  \big)\bigg)$$ for all but $O(x/\psi(x))$ integers up to $x$.

We will also turn our attention to finding an asymptotic formula involving iterates involving $\lambda$ and $\phi$. Banks, Luca, S\u{a}id\u{a}k, and Stanic in \cite{BLSS}  showed that for almost all $n$,

$$\lambda(\phi(n))=n \exp(-(1+o(1))(\log\log n)^2 \log\log\log n) \text{ and }$$
$$\phi(\lambda(n))=n \exp(-(1+o(1))(\log\log n) \log\log\log n).$$
As a corollary to Theorem \ref{MainTheorem} we will obtain asymptotic formulas for higher iterates involving $\lambda$ and $\phi.$ Specifically we prove the following.

\begin{theorem}\label{Second}
For $l\ge 0$ and $k \ge 1$, let $g(n)=\phi_l(\lambda(f(n)))$, where $f(n)$ is a $(k-1)$ iterated arithmetic function consisting of iterates of $\phi$ and $\lambda$. Then the normal order of  $\log(n/g(n))$ is $\frac{1}{(k-1)!}(\log\log{n})^{k}\log\log\log{n}.$
\end{theorem} An example of the use of this theorem is for $\phi\phi\lambda\phi\phi\lambda\lambda\phi(n).$ Since $l=2,k=5,$ we get that the normal order of $\log \frac{n}{\phi\phi\lambda\phi\phi\lambda\lambda\phi(n)}$ is 
$$\frac{1}{4!}(\log\log n)^5 \log\log\log n.$$

The proof of Theorem \ref{MainTheorem} involves breaking down $\frac{n}{\lambda_{k}(n)}$ in terms of the iterated Euler $\phi$ function by using 
\begin{equation}\label{AA}\frac{n}{\lambda_k(n)}=\bigg(\frac{n}{\phi(n)}\bigg)\bigg(\frac{\phi(n)}{\phi_2(n)}\bigg)\dots \bigg(\frac{\phi_{k-1}(n)}{\phi_{k}(n)}\bigg)\bigg(\frac{\phi_{k}(n)}{\lambda_{k}(n)}\bigg)\end{equation} of which estimates for all but the last term are known. Hence $\log \frac{n}{\lambda_{k}(n)}$ can be written as a sum of the logarithms on the right side of \eqref{AA} and so we'll analyze the term $\log(\phi_k(n)/\lambda_k(n)).$ The following notations and conventions will be used throughout the paper. The letters $p,q,r$ will always denote primes and $k \ge 2$ will be a fixed integer. Note that the theorem has already been proven for $k=1.$ Let $v_{p}(n)$ be the largest power of $p$ which divides $n,$ so that
$$n= \prod_{p}p^{v_{p}(n)}.$$ Let the set $\p_n$ be $\{p : p \equiv 1 \pmod{n} \}.$ Throughout the paper we will assume $x>e^{e^e}$ and $y=y(x)=\log\log x.$ Also let $\psi(x)$ be any function going to $\infty$ such that $\psi(x)=o(\log y)=o(\log\log\log x).$ Whenever we use the phrase  ``for almost all $n\le x$'' in a result, we mean that the result is true for all $n \le x$ except a set of size $O(x/\psi(x)).$ Lastly we note that any implicit constant may depend on $k.$

\section{Required Estimates}

The following estimates will be used throughout the paper. We use the Chebeshev bound

\begin{equation}\label{Cheb}
\sum_{n\le x} \Lambda(n) = \sum_{p\le x}\log p \ll x
\end{equation} where $\Lambda(n)$ is the von--Mangoldt function. We also require a formula of Mertens (See \cite[Theorem 2.7(b)]{MV}) 

\begin{equation} \label{Merten}
\sum_{q \le x} \frac{\log q}{q} = \log x + O(1).
\end{equation}
Using partial summation on \eqref{Cheb} we can obtain the tail estimates

\begin{equation} \label{logq_q2}
\sum_{q>x} \frac{\log q}{q^{2}} \ll \frac{1}{x}
\end{equation} and

\begin{equation} \label{1_q2}
\sum_{q>x} \frac{1}{q^{2}} \ll \frac{1}{x \log x}.
\end{equation}
Given $m,x$, let $A$ be the smallest $a$ for which $m^a>x.$ We can then manipulate the sums
$$\sum_{a \in \nat} \frac{P(a)}{m^{a}}=\frac{1}{m}\sum_{a=0}^{\infty} \frac{P(a)}{m^{a}} \text{ and } \sum_{\substack{a \in \nat \\ m^{a}>x}}\frac{P(a)}{m^{a}} \ll \frac{1}{x}\bigg| \sum_{a=0}^{\infty} \frac{P(a)}{m^{a-A}}\bigg|=\frac{1}{x}\bigg| \sum_{a=A}^{\infty} \frac{Q(a)}{m^{a}}\bigg|$$ for $Q(x)=P(x+A).$ Then by noting that $\sum_{a=A}^{\infty} \frac{P(a)}{m^{a}}\ll_P 1$ uniformly for $m\ge 2$ and $A \ge 0$ we obtain the estimates
\begin{equation} \label{geometric}
\sum_{a \in \nat} \frac{P(a)}{m^{a}} \ll_{P} \frac{1}{m}, \sum_{\substack{a \in \nat \\ m^{a}>x}} \frac{P(a)}{m^{a}} \ll_{P} \frac{1}{x}.
\end{equation}
From \cite[Corollary 1.15]{MV} we get
\begin{equation}\label{Recip1}
\sum_{s \le x} \frac{1}{s} = \log x + O(1)
\end{equation} from which it easily follows that
\begin{equation}\label{Recip2}
\sum_{\substack{D \le s \le x \\ s \equiv a \pmod{C}}} \frac{1}{s} \ll \frac{1}{D}+\frac{\log x}{C}.
\end{equation}
We will also make frequent use of the Brun-Titchmarsh inequality \cite[Theorem 3.9]{MV}

\begin{equation} \label{BT1}
\pi(t;n,a) \ll \frac{t}{\phi(n)\log(t/n)}.
\end{equation}
By partial summation on \eqref{BT1} we can obtain
\begin{equation} \label{BT2}
\sum_{\substack{p \le t \\ p \in \p_n}} \frac{1}{p} \ll \frac{\log\log t}{\phi(n)}.
\end{equation}
Whenever $n/\phi(n)$ is bounded, as it will be whenever $n$ is a prime, prime power or a product of two prime powers, we can replace this bound with
\begin{equation} \label{BT}
\sum_{\substack{p \le t \\ p \in \p_n}} \frac{1}{p} \le \frac{c\log\log t}{n}
\end{equation}
for some absolute constant $c$. We include the $c$ because occasionally we require an inequality as opposed to an estimate. We will also require the following asymptotic from \cite[Theorem 1]{P}
\begin{equation}
\sum_{\substack{p \in \p_{n} \\ p \leq t}} \frac{1}{p}=\frac{\log \log t}{\phi(n)}+O\bigg(\frac{\log n}{\phi(n)}\bigg),
\end{equation}
which easily implies that
\begin{equation}\label{BT3}
\sum_{\substack{p \in \p_{n} \\ p \leq t}} \frac{1}{p-1}=\frac{\log \log t}{\phi(n)}+O\bigg(\frac{\log n}{\phi(n)}\bigg),
\end{equation}
since the difference is

\begin{align*}\sum\limits_{\substack{p \in \p_{n} \\ p \leq t}} \frac{1}{p(p-1)} \le \sum_{m=1}^{\infty}\frac{1}{mn(mn+1)}  < \frac{1}{n^2}\sum_{m=1}^{\infty}\frac{1}{m^2} \ll \frac{1}{n^2}.\end{align*}

\section{Required Propositions and Proof of Theorem \ref{MainTheorem}}\label{Props}

As mentioned previously, the main contribution to $\log(n/\lambda_k(n))$ will come from $\log(\phi_k(n)/\lambda_k(n)).$ Finding this term will involve a summation over prime powers which divide each of $\phi_k(n)$ and $\lambda_k(n).$ It turns out that the largest contribution to this term will come from small primes which divide $\phi_k(n).$ By small, we mean primes $q\le(\log\log x)^k=y^k.$ Hence we will split the sum into small primes and large primes $q>y^k.$ Therefore to prove Theorem \ref{MainTheorem} we will require the following propositions. The first summations deal with the large primes which divide $\phi_k(n)$ and the second involves the large primes whose prime powers divide $\phi_k(n).$ We will show that the contribution of these primes to the main sum is small and hence it will end up as part of the error term.

\begin{proposition}\label{prop1} $$\sum_{\substack{q>y^{k} \\ \nu_{q}(\phik)=1}} (\nu_{q}(\phik)-\nu_{q}(\lamk))\log{q} \ll y^{k}\psi(x)$$ for almost all $n \leq x$. \end{proposition}

\begin{proposition}\label{prop2} $$\sum_{\substack{q > y^{k}\\ \nu_{q}(\phik) \geq 2}} \nu_{q}(\phik) \log{q} \ll y^{k}\psi(x)$$ for almost all $n \leq x.$ \end{proposition}

Since the main contribution will come from small primes dividing $\phi_k(n)$, the next propostion will show that the contribution of small primes dividing $\lambda_{k}(n)$ to the main sum can also be merged into the error term.

\begin{proposition}\label{prop3} $$\sum_{q\leq y^{k}} \nu_{q}(\lamk) \log{q} \ll y^{k}\psi(x)$$ for almost all $n \leq x.$ \end{proposition}

That will leave us with the contribution of small primes dividing $\phi_k(n).$ We will use an additive function to approximate this sum. Let $h_{k}(n)$ be the additive function defined by $$\hk(n) = \sum_{p_{1}\mid n} \sum_{p_{2}\mid p_{1}-1} \dots \sum_{p_{k} \mid p_{k-1}-1} \sum_{q \leq y^{k}} \nu_{q}(p_{k}-1)\log q.$$ The following propostion shows that the difference between the sum involving the small primes dividing $\phi_k(n)$ and the term $h_k(n)$ is small.

\begin{proposition}\label{prop4}  $$\sum_{q \leq y^{k}} \nu_{q}(\phik)\log{q} = \hk(n) + O(y^{k-1}\log y \cdot \psi(x))$$ for almost all $n \leq x$, \end{proposition} That leaves us with $\log(\phi_k(n)/\lambda_k(n))$ being approximated by $h_k(n).$ The last proposition will obtain an asymptotic formula for $h_k(n).$ From there we will have enough armoury to tackle Theorem \ref{MainTheorem}.

\begin{proposition}\label{hkestimate} $$\hk(n) = \frac{1}{(k-1)!} y^{k} \log y + O(y^{k})$$ for almost all $n \leq x.$ \end{proposition}

\begin{proof}[Proof of Theorem \ref{MainTheorem}] We start by breaking down the function $\log (n/\lambda_{k}(n)).$
$$\log\bigg(\frac{n}{\lambda_{k}(n)}\bigg)=\log\bigg(\frac{n}{\phi(n)}\bigg)+\log\bigg(\frac{\phi(n)}{\phi_{2}(n)}\bigg)+\dots+\log\bigg(\frac{\phi_{k-1}(n)}{\phi_{k}(n)}\bigg)+\log\bigg(\frac{\phi_{k}(n)}{\lambda_{k}(n)}\bigg).$$ Using the lower bound $\phi(m) \gg m/\log\log m,$ see \cite[Theorem 2.3]{MV} we have that
$$\log\bigg(\frac{n}{\phi(n)}\bigg)+\log\bigg(\frac{\phi(n)}{\phi_{2}(n)}\bigg)+\dots+\log\bigg(\frac{\phi_{k-1}(n)}{\phi_{k}(n)}\bigg) \ll \log\log\log n$$ and so
$$\log\bigg(\frac{n}{\lambda_{k}(n)}\bigg)=\log\bigg(\frac{\phi_{k}(n)}{\lambda_{k}(n)}\bigg)+O(\log\log\log n).$$ In fact we could have used a more precise estimate for $\phi_{i}(n)/\phi_{i+1}(n)$ for $i\ge 1$ which can be found in \cite{EGPS} but the one we used is more than good enough. Next we break down the remaining term into summations. We will break it up into small primes and large primes.
\begin{align*}
\log\bigg(\frac{\phi_{k}(n)}{\lambda_{k}(n)}\bigg) &= \sum_{\substack{q>y^{k}}} (\nu_{q}(\phik)-\nu_{q}(\lamk))\log{q} + \sum_{\substack{q \le y^{k}}} (\nu_{q}(\phik)-\nu_{q}(\lamk))\log{q} \\
& = \sum_{\substack{q>y^{k} \\ \nu_{q}(\phik)=1}} (\nu_{q}(\phik)-\nu_{q}(\lamk))\log{q} + \sum_{\substack{q>y^{k} \\ \nu_{q}(\phik)\ge 2}} (\nu_{q}(\phik)-\nu_{q}(\lamk))\log{q} \\ & \qquad + \sum_{q\leq y^{k}} \nu_{q}(\phi_{k}(n)) \log{q}-\sum_{q\leq y^{k}} \nu_{q}(\lamk) \log{q}.
\end{align*}
Note that if $a \mid b$, then $\lambda(a) \mid \phi(b)$ since $\lambda(a) \mid \phi(a) \mid \phi(ma)$ for any $m$. This quickly implies that $\lamk$ always divides $\phik$ for all $k$ and so we get 
$$0\le \sum\limits_{\substack{q>y^{k} \\ \nu_{q}(\phik)\ge 2}} (\nu_{q}(\phik)-\nu_{q}(\lamk))\log{q} \le \sum\limits_{\substack{q>y^{k} \\ \nu_{q}(\phik)\ge 2}} (\nu_{q}(\phik)\log{q}.$$ Using Propositions \ref{prop1},\ref{prop2},\ref{prop3} and \ref{prop4} we get
$$\log\bigg(\frac{n}{\lambda_{k}(n)}\bigg) = \hk(n)+O\bigg(y^{k}\psi(x)  \bigg)$$ for almost all $n \le x$. Finally by using Proposition \ref{hkestimate} we get 
$$\log\bigg(\frac{n}{\lambda_{k}(n)}\bigg) = \frac{1}{(k-1)!} y^{k} \log y + O\bigg(y^{k}\psi(x)  \bigg)$$ for almost all $n \le x,$ finishing the proof of Theorem \ref{MainTheorem}.
\end{proof}

\section{Prime Power Divisors of $\phi_{k}(n)$}\label{Intro}

For various reasons thoughout this paper, we are concerned with the number of $n \le x$ such that $q^{a}$ can divide $\phik$. We will analyze a few of those situations here:

Case 1: $q^{2} \mid n.$ Clearly the number of such $n$ is at most $\frac{x}{q^{2}}.$

Case 2: There exists $p_{1} \in \p_{q^{2}},p_{2} \in \p_{p_{1}},p_{3} \in \p_{p_{2}},...,p_{l} \in \p_{p_{l-1}}$ where $p_{l} \mid n.$ By using \eqref{BT} repeatedly we get that the number of such $n$ is
\begin{align*}\sum_{n \leq x}  \sum_{p_{1} \in \p_{q^{2}}}  \sum_{p_{2} \in \p_{p_{1}}}...\sum_{\substack{{p_{l} \in \p_{p_{l-1}}}\\ p_{l}|n}} 1 &= \sum_{p_{1} \in \p_{q^{2}}}  \sum_{p_{2} \in \p_{p_{1}}}...\sum_{\substack{{p_{l} \in \p_{p_{l-1}}}\\ n \le x  \\ p_{l}|n}} 1 \\
&\ll \sum_{p_{1} \in \p_{q^{2}}}  \sum_{p_{2} \in \p_{p_{1}}}...\sum_{{p_{l} \in \p_{p_{l-1}}}} \frac{x}{p_{l}} \\
&\ll \sum_{p_{1} \in \p_{q^{2}}}  \sum_{p_{2} \in \p_{p_{1}}}...\sum_{{p_{l-1} \in \p_{p_{l-2}}}} \frac{xy}{p_{l-1}} \\
&\ll \sum_{p_{1} \in \p_{q^{2}}}  \sum_{p_{2} \in \p_{p_{1}}} \frac{xy^{l-2}}{p_{2}} \\
&\ll \sum_{p_{1} \in \p_{q^{2}}} \frac{xy^{l-1}}{p_{1}} \\
&\ll \frac{xy^{l}}{q^{2}}\end{align*}
Now that we've taken care of any case where $p \in \p_{q^{2}}$, we are just left with the possibilities not containing any powers of $q$. Unfortunately these cases still allow for many possibilities which we will display in an array. There are lots of ways for a prime power $q^a$ to arise in $\phi_k(n)$ we now define various sets of primes that are involved in generating these powers of $q$, and we will eventually sum over all possibilities for these sets of primes.  The set $\LL_{h,i}$ will denote a finite set of primes. To begin, the set $\LL_{1,2}$ will be an arbitrary finite set of primes in $\p_q$ and let $\LL_{1,1}$ be empty.  That is:

Case 3: \\

Level (1,2)
$$\LL_{1,2} \subseteq \p_q.$$

Level (2,1) (Obtaining the primes in the previous level)

$\LL_{2,1}$ is any set of primes with the property that for all $p \in \LL_{1,1}\cup \LL_{1,2},$ there exists a unique prime $r \in \LL_{2,1}$ such that $r \in \p_p.$ In other words $p$ will divide $\phi(r)$ and hence the primes in $\LL_{2,1}$ will create the primes in $\LL_{1,1}\cup \LL_{1,2}.$

Level (2,2) (New primes in $\p_q$)
$$\LL_{2,2} \subseteq \p_q.$$ In general for all $1 < h \le k$ we have for all $p \in \LL_{h-1,1}\cup \LL_{h-1,2} $ there exists a unique prime $r \in \LL_{h,1}$ such that $r \in \p_p, \LL_{h,2}$ is an arbitrary subset of $\p_q$, and 
$$r\in \LL_{k,1}\cup\LL_{k,2} \Rightarrow r \mid n.$$

Some description of the terms are in order including some helpful definitions.

\begin{definition}
An incarnation $I$ of Case 3 is some specified description of how the primes  in a lower level create the primes in the level directly above.
\end{definition}
For example, for $k=3$, an incarnation $I$ for which $q^4 \mid \phi_3(n)$ would be $s_1,s_2,s_3,r_3,r_4 \in \p_q$ where $r_1 \in \p_{s_1},r_2\in\p_{s_2s_3}, p_1 \in \p_{r_1r_2},p_2 \in \p_{r_3r_4},$  with $p_1p_2\mid n.$

\begin{definition}
An subincarnation of $I$ is an incarnation with added conditions. In other words if $J$ is a subincarnation of $I$ and an integer $n$ satisfies incarnation $J,$ then it will also satisfy incarnation $I.$
\end{definition}
For example, $I$ is a subincarnation of the incarnation $s_1,s_3,r_3,r_4 \in \p_q$ where $r_1 \in \p_{s_1},r_2\in\p_{s_3}, p_1 \in \p_{r_1r_2},p_2 \in \p_{r_3r_4},$  with $p_1p_2\mid n.$ 

Let $p$ be a prime in $\LL_{h,i}$ which we need to divide $\phi_{k-h+1}(n)$. The definition of $\LL_{h,i}$ ensures that there is a unique prime dividing $\phi_{k-h}(n)$ for which $p \mid r-1$. The primes in levels $(k,1),(k,2)$ dividing $n$ are for the base case of the recursion, so that each prime divides $\phi_0(n)=n$. When $i=2$ we are introducing new primes to get greater powers of $q$ in $\phik$. Note that it's not necessary to have any primes on the levels $(i,2).$ In fact the ``worst case scenario" that we will see has no primes on these except Level (1,2).

Now that we've described the way to get $q^{a} \mid \phik$, what is our exponent $a?$  Let $m_{h,i}=\#\LL_{h,i}.$ From the recursion above we can see that $q^{m_{k,2}}\mid \phi(n)$ and so do the primes in $\LL_{k-1,1}.$
For the second iteration of $\phi$, $q^{m_{k,2}-1+m_{k-1,2}}\mid \phi_2(n)$ and so do the primes in $\LL_{k-2,1}.$ Hence the power of $q$ which divides $\phi_{k}(n)$ is
\begin{equation}\label{H}\max_{1 \le j \le k} (m_{1,1}+\sum_{2 \le h \le j} (m_{h,2}-1))\end{equation}
where the sum can be empty if there are no primes in the second level $(j,2)$ or there are not enough to survive, i.e. $m_{j,2}<j-1$ and hence $q \nmid \phi_j(\prod_{\LL_{j,2}}p).$ Without loss of generality, we can assume the former, since the later is a subincarnation of the former. 

Now we'll introduce some notation to be used in future propositions. For any single incarnation of Case 3, let $M$ be the total number of primes, $N$ be the total new primes introduced at the levels $(h,2)$ and $H$ be the maximum necessary level $(h,2).$ Specifically 

$$M= \sum_{h}(m_{h,1}+m_{h,2})~~N=\sum_{h\le H}m_{h,2}$$ and $H$ yields the maximum value in \eqref{H}. Note that under this notation, $q^{N-H+1}\mid \phik.$ For example, in the incarnation $I$ above, $$\LL_{1,2}=\{s_1,s_2,s_3\},\LL_{2,1}=\{r_1,r_2\},\LL_{2,2}=\{r_3,r_4\},\LL_{3,1}=\{p_1,p_2\},\LL_{3,2}=\emptyset$$ as well as $$m_{1,2}=3,m_{2,1}=2,m_{2,2}=2,m_{3,1}=2,m_{3,2}=0.$$ Hence $M=9,N=5,H=2$ and so the power of $q$ which divides $\phi_3(n)$ is $5-2+1=4$ as expected.

Now that we've described Case 3, how many possible $n$ are in that case?

\begin{lemma}\label{NumberOfCases}
The number of $n \le x$ satisfying any incarnation of Case 3 is
$$O\bigg(c^M\frac{xy^{M}}{q^{N}}\bigg)$$
where $c$ is the constant from equation \eqref{BT}. 
\end{lemma}
\begin{proof}
Let $\LL_h = \LL_{h,1}\cup\LL_{h,2}.$ We use Brun-Titchmarsh \eqref{BT} for all the primes at each level of Case 3, so the number of $n$ is

\begin{align*}\sum_{n \leq x}  \sum_{p_1 \in \LL_1}  \sum_{p_2 \in \LL_2}\dots\sum_{p_k \in \LL_k} 1 &= \sum_{p_1 \in \LL_1}  \sum_{p_2 \in \LL_2}\dots \sum_{p_{k} \in \LL_{k}} \sum_{\substack{p_{k} \mid n \\ n \le x}}1 \\
&\ll \sum_{p_1 \in \LL_1}  \sum_{p_2 \in \LL_2}\dots \sum_{p_{k} \in \LL_{k}}\frac{x}{\prod_{p_{k} \in \LL_{k}}p_k}.\end{align*} Note that we have repeatedly counted the same primes in the sum as we can reorder the primes in each level. It won't be important here, but will need to be more carefully addressed later. 
Since the primes in level $(k,1)$  gave us some $p_k \in \p_{p_{k-1}}$ for all the primes in $\LL_{k-1}$, and for $p \in \LL_{k,k}$ we have $p \in \p_{q}.$ By Brun--Titchmarsh \eqref{BT} we get that the above sum is
$$\ll \sum_{p_1 \in \LL_1}  \sum_{p_2 \in \LL_2}\dots \sum_{p_{k-1} \in \LL_{k-1}}\frac{x(cy)^{m_{k,1}+m_{k,2}}}{\prod_{p_{k-1} \in \LL_{k-1}}p_{k-1}q^{m_{k,2}}}.$$
Once again we get $m_{k-1,1}+m_{k-1,2}$ new applications of Brun-Titchmarsh giving the new primes in level $k-2$ as well as $m_{k-1,2}$ new powers of $q$. Continuing along in this manner we get:

\begin{align*}&\ll \sum_{p_1 \in \LL_1}  \frac{x(cy)^{\sum_{2 \leq i \le k}(m_{i,1}+m_{i,2})}}{\prod_{p_{1} \in \LL_{1}}p_1q^{\sum_{2 \leq i \le k}m_{i,2}}} \\
&\ll \frac{x(cy)^{\sum_{1 \leq i \le k}(m_{i,1}+m_{i,2})}}{q^{\sum_{1 \leq i \le k}m_{i,2}}} = \frac{x(cy)^{M}}{q^{N}}.\end{align*}
\end{proof}

The last thing we'll consider in this section about the ways to obtain $\phi_k(n)$ is to determine the number of possible incarnations of Case 3. We note that there are lots of incarnations which are subincarnations of others. We will develop a concept of minimality. 

\begin{definition}An incarnation of Case 3 is minimal if it does not contain any strings of $p_1 \in \p_{p_2},p_2\in \p_{p_3} \dots p_{k-1} \in \p_{p_k}$ where $p_k \mid n$.
\end{definition}
Note that any incarnation of Case 3 is a subincarnation of a minimal one. We now use this concept to show the number of necessary incarnations of Case 3 is small.

\section{Large Primes Dividing $\phi_k(n)$}\label{large}
In this section we will prove the two propostions dealing with $q$ being large. We'll start with the propostion where $\nu_{q}(\phik)=1.$
\begin{proof}[Proof of Proposition \ref{prop1}]
It suffices to show 
$$\sum_{n \leq x}\sum_{\substack{q>y^{k} \\ \nu_{q}(\phik)=1}}  (\nu_{q}(\phik)-\nu_{q}(\lamk))\log{q} \ll xy^{k}$$ as then there are at most $O\big(\frac{xy^{k}}{y^{k}\psi(x)}\big)= O\big(\frac{x}{\psi(x)}\big)$ such $n$ where the bound for the sum in Proposition \ref{prop1} fails to hold. We examine the cases where  $\nu_{q}(\phik)=1$. 
Using the notation in Lemma~\ref{NumberOfCases} we have two subcases for Case 3, whether $N=1$ or $N>1$.

Suppose $N=1$, then $H=1$, $m_{1,2}=1$ and $m_{h,2}=0$ for $1<h \le k$. Since $m_{h,1}\le m_{h-1,1}+m_{h-1,2}$ we get $m_{h,1} \le 1$ for all $1 \le h \le k$. Hence $m_{h,1} = 1$ for all $h \le k$ and so we get the case:

$$p_{1} \in \p_{q},p_{2} \in \p_{p_{1}},p_{3} \in \p_{p_{2}},\dots,p_{k} \in \p_{p_{k-1}}$$
where $p_{k} \mid n$.
However in this case we also get $\nu_{q}(\lamk))=1$ giving us no additions to our sum.

Suppose $N>1$, then $M= \sum_{h}(m_{h,1}+m_{h,2}) \le k\sum_{h}m_{h,2} = kN$  so the number of cases we get are 

$$O\bigg(c^M\frac{xy^{M}}{q^{N}}\bigg) \ll \frac{c^Mxy^{kN}}{q^{N}} \ll \frac{c^Mxy^{2k}}{q^{2}}$$ 
since $y > q^{k}.$ Since $v_q(\phi_k(n))=N-H+1$ and $H\le k$, $N\le k$ implying that $M\le k^2$.Hence $c^M$ is bounded as a function of $k.$ Also since $M$ is bounded in terms of $k,$ there are $O_k(1)$ possible incarnations of Case 3, and the bound already absorbs the possiblities from Cases 1 and 2. Hence we have

\begin{align*}\sum_{q>y^{k}}\sum_{\substack{n \leq x \\ \nu_{q}(\phik)=1}}  (\nu_{q}(\phik)-\nu_{q}(\lamk))\log{q} &\leq \sum_{q>y^{k}}\sum_{\substack{n \leq x \\ \nu_{q}(\phik)=1 \\ N>1}} \log{q} \\
&\ll  \sum_{q>y^{k}} \frac{xy^{2k}\log q}{q^{2}} \\ &\ll xy^{k}\end{align*} by \eqref{logq_q2}.
\end{proof}

We turn our attention to $v_q(\phi_k(n))>1.$ We have to be more careful here since we can't guarantee that the number of incarnations of Case 3 is $O_k(1).$ We'll start by proving a lemma which can eliminate a lot of those cases.

\begin{lemma}\label{Eliminate}
Let $q>y^k$ and $S_q=S_{q}(x)$ consist of all $n \le x$ such that Case 1,2 or Case 3 where $M \le k(N-1)$ occurs. Then 
$$\#S_q \ll \frac{xy^k}{q^2}$$
\end{lemma}
\begin{proof}
There are clearly $O_k(1)$ incarnations of Cases 1 and 2 and each yield at most $O(xy^k/q^2)$ such $n.$ By Lemma \ref{NumberOfCases} for each incarnation of Case 3,  we get at most
$$O\bigg(\frac{c^My^M}{q^N}\bigg) \ll \frac{c^My^k}{q^2}$$ such $n$ since $M\le k(N-1)$ and $q>y^k.$ It remains to show we only require $O_k(1)$ such incarnations. Suppose $n$ satisfies an incarnation with $M\le k(N-1)$. Then it also satisfies a minimal incarnation with $M\le k(N-1)$ since removing a string of $p_1 \in \p_{p_2},p_2\in \p_{p_3} \dots p_{k-1} \in \p_{p_k}$, would decrease $N$ by $1$ and $M$ by $k$ leaving the inequality unchanged. Secondly we can assume that $n$ also satisfies an incarnation where $k(N-2) < M \le k(N-1)$ since we can keep eliminating primes in the $\LL_{i,2},$ which decrease $N$ by $1$, but $M$ by at most $k.$ This must eventually produce an incarnation where $k(N-2) < M \le k(N-1)$ since if we eliminate all primes in the $\LL_{i,2}$ but $1,$ then $M>k(N-1).$ Also note that the condition $m_{h,1} \le m_{h-1,1}+m_{h-1,2}$ forces $M \le kN.$ If $M$ is bounded between $k(N-2)$ and $kN$ and the incarnation is minimal, we get that $N$ is bounded by $2k$ since eliminating a prime in $\LL_{i,2}$ can only shrink $M$ by at most $k-1$ since our incarnation is minimal.

Therefore $n$ satisifies an incarnation where $N$ and hence $M$ are bounded functions of k. Since there are only $O_k(1)$ such incarnations, we get our result, noting that $c^M$ can be absorbed into the constant as well.
\end{proof} 

\begin{proof}[Proof of Proposition \ref{prop2}]
Let $S=S(x) = \bigcup_{q>y^{k}} S_{q}$. Using Lemma \ref{Eliminate} we have 

$$\#S \le \sum_{q>y^{k}} \# S_{q} \ll  \sum_{q>y^{k}}\frac{xy^{k}}{q^{2}}
\ll xy^{k}\sum_{q>y^{k}}\frac{1}{q^{2}} \ll \frac{xy^{k}}{\log(y^{k})y^{k}} \ll \frac{x}{\psi(x)} 
$$ by \eqref{1_q2}. As for the $n$ with $n \notin S$ and $a = \nu_{q}(\phik) >1,$ the only remaining case is that $M > k(N-1)$. Recall that $a = N+H-1.$ If $H=1$, then $N=m_{1,2}=a,$ and so $m_{2,1}=a-1$ or $a$. Otherwise for $k \ge 2,$

$$M= \sum_{h}m_{h,1} \leq a+ (k-1)m_{2,1} \le a+(k-1)(a-2) = k(a-1) - k + 2 \le (k-1)N$$ leading to a contradiction. If $H>1$, then we again wish to show that $m_{2,1} \ge a-k.$
\begin{align*}M &= \sum_{h}(m_{h,1}+m_{h,2}) \\ & \le km_{1,2} + (k-1)\sum_{h>1}m_{h,2} \\ &= m_{1,2} + (k-1)N \\ &= k(N-1) -N + k +m_{1,2}\end{align*}
which implies $m_{1,2} > N-k$  and so $\sum_{h>1}m_{h,2} = N - m_{1,1} < k.$ Therefore if $m_{2,1} < a-k,$ then
\begin{align*}M &= \sum_{h}(m_{h,1}+m_{h,2}) \\ &\le m_{1,2}+(k-1)m_{2,1} + (k-1)\sum_{h>1}m_{h,2} \le  a+(k-1)(a-k-1) + (k-1)(k-1)
\\ &= ak -2k \\ &\le k(N-1)\end{align*} as $N>a$ again leading to a contradiction. Hence $m_{2,1} \ge a-k$ and so we can get

\begin{align*}\sum_{\substack{n \notin S \\ n \leq x}}\sum_{\substack{q>y^{k} \\ \nu_{q}(\phik)>1}}  (\nu_{q}(\phik)\log{q} & \le 2\sum_{\substack{n \notin S \\ n \leq x}}\sum_{\substack{q>y^{k} \\ \nu_{q}(\phik)>1}}  (\nu_{q}(\phik)-1)\log{q} \\ 
& \ll \sum_{\substack{q>y^{k}}}\log q\sum_{a \ge 2}a\sum_{\substack{n \leq x \\ n \notin S \\ \nu_{q}(\phik)=a }}1.
\end{align*}
Unfortunately, just blindly using the Brun-Titchmarsh inequality in \eqref{BT} won't be good enough as we must sum over all $a.$ Let $g(a,k)=(a-k)!$ if $a \ge k$ or $1$ otherwise and note that since we have $m_{1,2}\ge a-k$, we have at least $g(a,k)$ permutations of the same primes. Then by using Lemma \ref{NumberOfCases} we get

\begin{align*}a\sum_{\substack{q>y^{k}}}\log q\sum_{\substack{n \leq x \\ n \notin S \\ \nu_{q}(\phik)=a }}1 & \ll a\frac{x(cy)^{M}}{q^{N}g(a,k)} \ll \frac{ac^{k(a+k-1)}xy^{2k}}{q^2g(a,k)} 
\end{align*} using the assumption that $q>y^k$ and $M \le kN \le k(a+k-1).$ Hence we get our sum is 

\begin{align*}\sum_{\substack{n \notin S \\ n \leq x}}\sum_{\substack{q>y^{k} \\ \nu_{q}(\phik)>1}}  (\nu_{q}(\phik)\log{q} & \ll \sum_{\substack{q>y^{k}}}\log q\sum_{a \ge 2}\frac{ac^{k(a+k-1)}xy^{2k}}{q^2g(a,k)} \\
& = xy^{2k}\sum_{q>y^{k}}\frac{\log q}{q^{2}}\sum_{a \ge 2}\frac{ac^{k(a+k-1)}}{g(a,k)}
\end{align*} However the latter sum converges to some function depending on $k$, and so we get

$$ \ll xy^{2k}\sum_{q>y^{k}}\frac{\log q}{q^{2}} \ll xy^{k}$$ by \eqref{logq_q2}.
\end{proof}

\section{Small Primes Dividing $\lambda_k(n)$}\label{lambda}
We now turn our attention to the bound involving $\lambda_k(n)$ in the summand. Just like when we were dealing with the number of cases where $q^a \mid \phi_k(n)$, we will need a lemma to deal with the number of cases where $q^a \mid \lambda_k(n).$ Fortunately this case is much simpler as the only two ways for $q^a \mid \lambda(n)$ is for $q^{a+1} \mid n$ or for there to exist $p \mid n$ with $p \in \p_{q^a}.$ Note that these conditions aren't sufficient, but are necessary when $q=2.$

\begin{lemma}\label{NumberOfCases2} The number of positive integers $n \le x$ for which $q^{a}\mid \lamk$ is $O(\frac{xy^{k}}{q^{a}}).$
\end{lemma}
\begin{proof}
We'll proceed by induction on $k$. If $k=1$, then $q^{a}\mid \lambda(n)$ if $q^{a+1}\mid n$ or $p \in \p_{q^{a}}$ with $p\mid n$. The number of such $n$ is at most

$$\sum_{\substack{n \le x\\ q^{a+1} \mid n }}1 +\sum_{\substack{n \le x\\ p \in \p_{q^{a}} \\ p \mid n }}1 \ll 
\frac{x}{q^{a+1}} +\sum_{p \in \p_{q^{a}}}\frac{x}{p} \ll \frac{x}{q^{a+1}}+\frac{xy}{q^{a}} \ll \frac{xy}{q^{a}}.$$ using \eqref{BT}.
Suppose the number of $n \le x$ for which $q^{a}\mid \lambda_{k-1}(n)$ is $O(\frac{xy^{k-1}}{q^{a}})$. If $q^{a}\mid \lamk$, then either $q^{a+1}\mid \lambda_{k-1}(n)$ or $p \in \p_{q^{a}}$ with $p \mid \lambda_{k-1}(n)$. Hence the number of such $n$ is bounded by

$$\sum_{\substack{n \le x\\ q^{a+1}\mid \lambda_{k-1}(n) }}1 +\sum_{\substack{n \le x\\ p \in \p_{q^{a}} \\ p|\lambda_{k-1}(n) }}1 \ll 
\frac{xy^{k-1}}{q^{a+1}} +\sum_{p \in \p_{q^{a}}}\frac{xy^{k-1}}{p} \ll \frac{xy^{k-1}}{q^{a+1}}+\frac{xy^{k}}{q^{a}} \ll \frac{xy^{k}}{q^{a}}$$ as needed.
\end{proof}

\begin{proof}[Proof of Proposition \ref{prop3}]
Like in the proof of previous propositions, we'll show $$\sum_{n \le x} \sum_{q\leq y^{k}} \nu_{q}(\lamk) \log{q} \ll xy^{k}.$$ The left hand side is equal to
\begin{align*}\sum_{n \le x} \sum_{q\leq y^{k}} \nu_{q}(\lamk) \log{q} &=\sum_{n \le x} \sum_{q\leq y^{k}} \log q \sum_{\substack{a \in \nat \\ q^{a}\mid \lamk}}1 \\
& \le \sum_{n \le x} \sum_{q\leq y^{k}} \log q \sum_{\substack{a \in \nat \\ q^{a} \le y^{k}}}1 + \sum_{n \le x} \sum_{q\leq y^{k}} \log q \sum_{\substack{a \in \nat \\ q^{a}\mid \lamk \\ q^{a} > y^{k}}}1.\end{align*}
The first sum is $$\sum_{n \le x} \sum_{q\leq y^{k}} \log q \sum_{\substack{a \in \nat \\ q^{a} \le y^{k}}}1 = \sum_{n \le x} \sum_{m\leq y^{k}} \Lambda(m) \ll \sum\limits_{n \le x} y^{k} \ll xy^{k},$$
and by Lemma \ref{NumberOfCases2} and using the geometric estimate in \eqref{geometric} the second sum becomes 

$$\sum_{n \le x} \sum_{q\leq y^{k}} \log q \sum_{\substack{a \in \nat \\ q^{a}\mid \lamk \\ q^{a} > y^{k}}}1 \ll \sum_{q\leq y^{k}} \log q \sum_{\substack{a \in \nat \\ q^{a} > y^{k}}}\frac{xy^{k}}{q^{a}} \ll \sum_{q\leq y^{k}} \log q \frac{xy^{k}}{y^{k}} \ll xy^{k}.$$
\end{proof}

\section{Reduction To $h_k(n)$ For Small Primes}\label{reduction}
The small primes dividing $\phi_k(n)$ are what contributes to the asymptotic term of $\log(n/\lambda_k(n))$. In this section we show that the important case is the supersquarefree case of $p$ dividing $\phi_{k}(n)$ which is when $$p \in \p_{p_1}, p_1 \in \p_{p_2} \dots p_{k-1} \in \p_{p_k}, p_k \mid n.$$ For this reason we will approximate the sum $\sum_{q \le y^k}v_q(\phi_k(n))\log q$ with \begin{equation}\label{hkDef}\hk(n) = \sum_{p_{1}\mid n} \sum_{p_{2}\mid p_{1}-1} \dots \sum_{p_{k} \mid p_{k-1}-1} \sum_{q \leq y^{k}} \nu_{q}(p_{k}-1)\log q.\end{equation}
\begin{proof}[Proof of Proposition \ref{prop4}]
For any fixed prime $q$, we know that $$v_{q}(\phi(m)) = \max \{ 0,v_{q}(m)-1 \} + \sum_{p\mid m}v_{q}(p-1),$$
which implies
 $$\sum_{p \mid m}v_{q}(p-1) \le v_{q}(\phi(m)) \le v_{q}(m) + \sum_{p\mid m}v_{q}(p-1).$$
Repeated use of this inequality for $m=\phi_{l}(n)$ where $l$ ranges from $k-1$ to $0$ yields

\begin{equation}\label{howdy}\begin{split}\sum_{p\mid \phi_{k-1}(n)}v_{q}(p-1) & \le v_{q}(\phik) \\ & \le \sum_{p\mid \phi_{k-1}(n)}v_{q}(p-1) + \sum_{p\mid \phi_{k-2}(n)}v_{q}(p-1) \\ & \qquad\qquad + \dots + \sum_{p\mid \phi(n)}v_{q}(p-1)+v_{q}(n).\end{split}
\end{equation}
A prime $p$ divides $\phi_{k-1}(n)$ either in the supersquarefree case (ssf), or not in the supersquarefree case (nssf), yielding

\begin{align*}\sum_{ssf}v_{q}(p-1) &\le \sum_{p\mid \phi_{k-1}(n)}v_{q}(p-1) \\
& \le \sum_{ssf}v_{q}(p-1) + \sum_{nssf}v_{q}(p-1).\end{align*}
Combining this inequality with \eqref{howdy} yields 
\begin{align*}\sum_{ssf} v_{q}&(p-1) \le v_{q}(\phik) \\ &\le \sum_{ssf}v_{q}(p-1) + \sum_{nssf}v_{q}(p-1) + \sum_{p\mid \phi_{k-2}(n)}v_{q}(p-1) + \dots + \sum_{p\mid \phi(n)}v_{q}(p-1)+v_{q}(n).\end{align*}
Subtracting the sum over the supersquarefree case, multiplying through by $\log q$ and summing over $q \le y^{k}$ we get

\begin{align*}0 & \le \sum\limits_{q \leq y^{k}} \nu_{q}(\phik)\log{q} - \hk(n) \\
& \le \sum_{q \le y^{k}}\sum_{nssf}v_{q}(p-1)\log q + \sum_{q \le y^{k}}\sum_{p|\phi_{k-2}(n)}v_{q}(p-1)\log q+ \dots +\sum_{q \le y^{k}}\sum_{p\mid n}v_{q}(p-1)\log q \end{align*} where we get $h_k(n)$ from \eqref{hkDef}.
Hence it suffices to show that the sum on the right side becomes our error term. For the sum

\begin{align*}\sum_{n \le x}\sum_{q \le y^{k}}\sum_{p \mid \phi_{m}(n)}v_{q}(p-1)\log q & = \sum_{n \le x}\sum_{q \le y^{k}}\sum_{p\mid \phi_{m}(n)}\sum_{\substack{a \in \nat \\ q^{a} \mid p-1}}\log q \\ 
&= \sum_{n \le x}\sum_{q \le y^{k}}\log q \sum_{a \in \nat}\sum_{\substack{p \in \p_{q^{a}} \\ p \mid \phi_m(n)}}1, \end{align*}
we'll split the sum over values of $p \le y^{k-1}$ and $p > y^{k-1}.$ For $p \le y^{k-1}$ we uniformly get for all $n$ that

\begin{align*}\sum_{q \le y^{k}}\log q \sum_{a \in \nat}\sum_{\substack{p \in \p_{q^{a}} \\ p \le y^{k-1} \\ p \mid \phi_m(n)}}1 & \le \sum_{q \le y^{k}}\log q \sum_{a \in \nat}\pi(y^{k-1};q^{a},1)\\
& \ll \sum_{q \le y^{k}}\log q \sum_{a \in \nat}\frac{y^{k-1}}{\phi(q^{a})} \\
& \ll y^{k-1}\sum_{q \le y^{k}}\frac{\log q}{q}\\
& \ll y^{k-1}\log y 
\end{align*}  using the geometric estimate \eqref{geometric} and the prime number theorem for arithmetic progressions. As for $p > y^{k-1}$ we fix an $M$ and $N$ from case 3 for which $ p\mid \phi_{m}(n)$, of which there are at most $O_k(1)$ such $M,N$ since $v_p(\phi(m))=1$. Therefore

\begin{align*}\sum_{n \le x}\sum_{q \le y^{k}}\log q \sum_{a \in \nat}\sum_{\substack{p>y^{k-1} \\ p \in \p_{q^{a}} \\ p \mid \phi_m(n)}}1 & \ll \sum_{q \le y^{k}}\log q \sum_{a \in \nat}\sum_{\substack{p \in \p_{q^{a}}\\p>y^{k-1}}}\frac{xy^{M}}{p^{N}} \\
& \le \sum_{q \le y^{k}}\log q \sum_{a \in \nat}\sum_{p \in \p_{q^{a}}}\frac{xy^{M-(k-1)(N-1)}}{p} \\
& \ll \sum_{q \le y^{k}}\log q \sum_{a \in \nat}\frac{xy^{M-(k-1)(N-1)+1}}{q^{a}} \\
& \ll \sum_{q \le y^{k}}\frac{xy^{M-(k-1)(N-1)+1}\log q}{q} \\
& \ll xy^{M-(k-1)(N-1)+1}\log y^{k} \\
& \ll xy^{M-(k-1)(N-1)+1}\log y. \end{align*}
Since the $M,N$ were chosen for $\phi_{m}(n)$ we know that $M \le mN$ where equality holds if and only if we are in the supersquarefree case. Now either $m \le k-2$  or $m=k-1$ and we are not in the supersquarefreecase. In the former case we have an error of 

$$O(xy^{(k-2)N-(k-1)(N-1)+1}\log y)=O(xy^{k-N}\log y) = O(xy^{k-1}\log y)$$ since $N \ge 1$, or in the latter case

$$O(xy^{(k-1)N-1-(k-1)(N-1)+1}\log y)=O(xy^{k-1}\log y).$$ Thus we get 

\begin{align*}\sum_{n \le x}\bigg(\sum_{q \le y^{k}}\sum_{nssf}v_{q}(p-1)\log q & + \sum_{q \le y^{k}}\sum_{p|\phi_{k-2}(n)}v_{q}(p-1)\log q+ \dots \\ & +\sum_{q \le y^{k}}\sum_{p|n}v_{q}(p-1)\log q\bigg) \ll xy^{k-1}\log y\end{align*} and so

\begin{align*}\sum_{q \le y^{k}}\sum_{nssf}v_{q}(p-1)\log q & + \sum_{q \le y^{k}}\sum_{p|\phi_{k-2}(n)}v_{q}(p-1)\log q+ \dots \\ & +\sum_{q \le y^{k}}\sum_{p|n}v_{q}(p-1)\log q \ll y^{k-1}\log y \cdot \psi(x)\end{align*} as required.
\end{proof}

\section{Reduction to the First and Second Moments}\label{Turan}
The Tur\'{a}n-Kubilius inequality \cite[Lemma 3.1]{K} asserts that if $f(n)$ is a complex additive function, then there exists an absolute constant $C$ such that
\begin{equation}\label{TK}\sum_{n\le x}\abs{f(n)-M_1(x)}^2 \le CxM_2(x)\end{equation} where $M_1(x)=\sum_{p \leq x}\abs{f(p)}/p$ and $M_2(x)=\sum_{p \leq x}\abs{f(p)}^2/p.$ Since $h_k(n)$ is additive we can apply this inequality where $M_{1}(x)= \sum_{p \leq x}\hk (p)/p$, $M_{2}(x)= \sum_{p \leq x}\hk (p)^{2}/p.$ We will need to find bounds on $M_1$ and $M_2$ therefore it's our goal to prove the following two propositions:

\begin{proposition}\label{M1} For all $x>e^{e^{e}},$ $$M_{1}(x) =  \frac{1}{(k-1)! } y^{k} \log y + O(y^{k})$$ \end{proposition}
\begin{proposition}\label{M2} For all $x>e^{e^{e}},$ $$M_{2}(x) \ll  y^{2k-1} \log^{k-1} y.$$  \end{proposition}
These will lead to a proof of Proposition \ref{hkestimate}.
\begin{proof}[Proof of Proposition \ref{hkestimate}]
Let $N$ denote the number of $n \leq x$ for which  $\abs{\hk(n) - M_{1}(x)} > y^{k}.$ The contribution of such $n$ to the sum in \eqref{TK} is at least $Ny^{2k}.$ Thus Propostion \ref{M2} implies $N \ll x \log^{k-1}y /y$ and so Proposition \ref{M1} implies that $\hk(n) = \frac{1}{(k-1)!} y^{k} \log y + O(y^{k})$ except for a set of size $O(x (\log y)^{k-1} /y).$
\end{proof}

\section{Lots of Summations}
In our proofs of Propositions \ref{M1} and \ref{M2} we will see that $M_1(x)$ and $M_2(x)$ will reduce to summations involving $\pi(x;p,1).$ We will be using some sieve techniques to bound these sums and those will require some bounds on sums on multiplicative functions involving $\phi(m).$  This section will involve the estimation of the latter sums.

\begin{lemma}\label{PhiSumLem}
For any non-negative integer L we have
\begin{equation} \label{PhiSum}
\sum_{m \le t} \frac{m^{L}}{\phi(m)^{L+1}} \ll_{L} \log t.
\end{equation}
\end{lemma}
\begin{proof}
If $f(n)$ is a non-negative multiplicative function, we know that
\begin{equation} \label{Multi}
\sum_{n \le t}f(n) \le \prod_{p \le t}\sum_{r=0}^\infty f(p^{r}).
\end{equation}
Applying \eqref{Multi} with $\frac{m^{L}}{\phi(m)^{L+1}}$ yields
\begin{align*}\sum_{m \le t} \frac{m^{L}}{\phi(m)^{L+1}} & \le  \prod_{p \le t}\bigg(1+\sum_{r=1}^\infty\frac{p^{rL}}{( p^{r}-p^{r-1})^{L+1}} \bigg) \\
& = \prod_{p \le t}\bigg(1+\sum_{r=1}^\infty\frac{p^{L-r+1}}{(p-1)^{L+1}} \bigg) \\
& = \prod_{p \le t}\bigg(1+\frac{1}{(p-1)^{L+1}}\frac{p^{L}}{1-\frac{1}{p}} \bigg) \\
& = \prod_{p \le t}\bigg(1+\frac{p^{L+1}}{(p-1)^{L+2}} \bigg) \\
& \le \exp\bigg( \sum_{p \le t}\log \bigg(1+\frac{p^{L+1}}{(p-1)^{L+2}} \bigg) \bigg)\\
& = \exp\bigg( \sum_{p \le t}\bigg( \frac{p^{L+1}}{(p-1)^{L+2}} + O_L\bigg(\frac{1}{p^{2}} \bigg) \bigg) \bigg)\\
& = \exp\bigg( \sum_{p \le t}\bigg(\frac{1}{p} + O_L\bigg(\frac{1}{p^{2}} \bigg) \bigg) \bigg)\\
& \ll_{L} \log t
\end{align*}
using \eqref{Merten}.
\end{proof}

\begin{lemma} \label{Sieve1}
Given a positive integer $C\le t^{\gamma}$ and non-negative integer $L$ we have
\begin{equation} \label{SieveSum1}
\sum_{m \le t} \frac{(Cm+1)^{L}}{\phi(Cm+1)^{L}\phi(m)} \ll_{L,\gamma} \log t.
\end{equation}
\end{lemma}
\begin{proof}
It will suffice to show 
$$\sum_{m \le t} \frac{(Cm+1)^{2L-1}}{\phi(Cm+1)^{2L}} \ll_{L} \frac{\log t}{C}$$
as then by Cauchy--Schwarz we can get that
\begin{align*}
\bigg(\sum_{m \le t} \frac{(Cm+1)^{L}}{\phi(Cm+1)^{L}\phi(m)}\bigg)^{2} &\le \sum_{m \le t} \frac{(Cm+1)^{2L-1}}{\phi(Cm+1)^{2L}}\sum_{m \le t} \frac{(Cm+1)}{\phi(m)^{2}} \\
& \ll_{L}\bigg(\frac{\log t}{C}\bigg)C \log t \\
& \ll_{L} \log^{2} t
\end{align*}
by using \eqref{PhiSum}. Using Mobius inversion, let $s(n)$ be the multiplicative function defined by
$$\frac{n^{2L}}{\phi(n)^{2L}} = 1 * s = \sum_{d \mid n}s(d).$$ Testing at prime powers, we can easily see that
$$s(1)=1,s(p)=\bigg(1-\frac{1}{p}\bigg)^{-2L}-1 \text{ and } s(p^{k})=0 \text{ for all } k\ge 2.$$
Hence 
\begin{align*}
\sum_{m \le t} \frac{(Cm+1)^{2L-1}}{\phi(Cm+1)^{2L}} &= \sum_{\substack{C<n \le Ct+1 \\ n \equiv 1 \pmod{C}}} \frac{n^{2L-1}}{\phi(n)^{2L}} \\
&= \sum_{\substack{C<n \le Ct+1 \\ n \equiv 1 \pmod{C}}}\frac{1}{n} \frac{n^{2L}}{\phi(n)^{2L}} \\
& = \sum_{\substack{C<n \le Ct+1 \\ n \equiv 1 \pmod{C}}}\frac{1}{n} \sum_{d \mid n} s(d) \\
& = \sum_{d\le Ct+1} s(d)\sum_{\substack{C<n \le Ct+1 \\ d \mid n \\ n \equiv 1 \pmod{C}}}\frac{1}{n} .
\end{align*}
By \eqref{Recip2} and noticing that $C$ and $d$ are relatively prime we get
\begin{align*}
\sum_{\substack{C<n \le Ct+1 \\ d \mid n \\ n \equiv 1 \pmod C}}\frac{1}{n} \ll \frac{1}{C+1}+ \frac{\log t}{dC}
\end{align*} where the first term occurs only if $d \mid C+1.$ We require some estimates on $s(d).$

\begin{align*}
\sum_{d\le Ct+1} \frac{s(d)}{d} &\le \prod_{p\le Ct+1}\bigg(1 + \frac{(1-1/p)^{-2L}-1}{p} \bigg) \\
& \le \prod_{p\le Ct+1}\bigg(1 + \frac{C_L}{p^2} \bigg) \\
& = \exp\bigg(\sum_{p\le Ct+1}\log\bigg(1 + \frac{C_L}{p^2}\bigg)\bigg) \\
& = \exp\bigg(\sum_{p\le Ct+1}O_L\bigg(\frac{1}{p^2}\bigg)\bigg) \\
& = \exp(O_L(1)) \\
& \ll_L 1
\end{align*}
and 
\begin{align*}
\sum_{\substack{d\le Ct+1 \\ d \mid C+1}} s(d) & \le \sum_{d \mid C+1} s(d) \\
& = (1*s)(C+1)\\
& = \bigg(\frac{C+1}{\phi(C+1)}\bigg)^{2L}\\
& \ll (\log\log C)^{2L}\\
&\ll_{\gamma} (\log\log t)^{2L}\\
&\ll_{L,\gamma}\log t.
\end{align*}
Therefore
$$\sum_{m \le t} \frac{(Cm+1)^{2L-1}}{\phi(Cm+1)^{2L}} \ll  \sum_{\substack{d\le Ct+1 \\ d\mid C+1}} \frac{s(d)}{C+1}+ \sum_{d\le t}\frac{s(d)\log t}{Cd} \ll_{L,\gamma} \frac{\log t}{C}$$ as needed.
\end{proof}

\begin{lemma} \label{Sieve2}
For positive integers $C_1,C_2,\dots,C_r\le t^{\gamma}$ and non-negative integers $L_1,L_2,\dots,L_r$ we have
\begin{equation} \label{SieveSum2}
\sum_{m \le t} \frac{(C_{1}m+1)^{L_{1}}(C_{2}m+1)^{L_{2}}\dots(C_{r}m+1)^{L_{r}}}{\phi(C_{1}m+1)^{L_{1}}\phi(C_{2}m+1)^{L_{2}}\dots\phi(C_{r}m+1)^{L_{r}}\phi(m)} \ll_{L_{1},\dots,L_{r},\gamma} \log t.
\end{equation}
\end{lemma}
\begin{proof}
We proceed by induction. The case $r=1$ is covered by Lemma \ref{Sieve1}.
Suppose 
$$\sum_{m \le t} \frac{(C_{1}m+1)^{L_{1}}(C_{2}m+1)^{L_{2}}...(C_{r}m+1)^{L_{r}}}{\phi(C_{1}m+1)^{L_{1}}\phi(C_{2}m+1)^{L_{2}}\dots\phi(C_{r}m+1)^{L_{r}}\phi(m)} \ll_{L_{1},\dots,L_{r},\gamma} \log t.$$
By Cauchy--Schwarz, we get that
\begin{align*}
\bigg( \sum_{m \le t} & \frac{(C_{1}m+1)^{L_{1}}(C_{2}m+1)^{L_{2}}\dots(C_{r+1}m+1)^{L_{r+1}}}{\phi(C_{1}m+1)^{L_{1}}\phi(C_{2}m+1)^{L_{2}}\dots\phi(C_{r+1}m+1)^{L_{r+1}}\phi(m)}\bigg)^2 \\
& \le \sum_{m \le t} \frac{(C_{1}m+1)^{2L_{1}}(C_{2}m+1)^{2L_{2}}\dots(C_{r}m+1)^{2L_{r}}}{\phi(C_{1}m+1)^{2L_{1}}\phi(C_{2}m+1)^{2L_{2}}\dots\phi(C_{r}m+1)^{2L_{r}}\phi(m)}\sum_{m \le t} \frac{(C_{r+1}m+1)^{2L_{r+1}}}{\phi(C_{r+1}m+1)^{2L_{r+1}}\phi(m)} \\
& \ll_{L_{1},\dots,L_{r+1},\gamma} \log^2 t
\end{align*}
by Lemma \ref{Sieve1}, completing the proof.
\end{proof}

\begin{lemma} \label{Sieve3}
For positive integers $C_1,C_2,...,C_r\le t^{\gamma}$ and non-negative integers $L_1,L_2,...,L_r,L$ we have
\begin{equation} \label{SieveSum3}
\sum_{m \le t} \frac{(C_{1}m+1)^{L_{1}}(C_{2}m+1)^{L_{2}}...(C_{r}m+1)^{L_{r}}m^{L-1}}{\phi(C_{1}m+1)^{L_{1}}\phi(C_{2}m+1)^{L_{2}}\dots\phi(C_{r}m+1)^{L_{r}}\phi(m)^{L}} \ll_{L_{1},\dots,L_{r},L,\gamma} \log t.
\end{equation}
\end{lemma}
\begin{proof}
Once again we'll use Cauchy--Schwarz and the previous lemmas.
\begin{align*}
\bigg(\sum_{m \le t} & \frac{(C_{1}m+1)^{L_{1}}(C_{2}m+1)^{L_{2}}\dots(C_{r}m+1)^{L_{r}}m^{L-1}}{\phi(C_{1}m+1)^{L_{1}}\phi(C_{2}m+1)^{L_{2}}\dots\phi(C_{r}m+1)^{L_{r}}\phi(m)^{L}}  \bigg)^2 \\
& \le \sum_{m \le t}\frac{(C_{1}m+1)^{2L_{1}}(C_{2}m+1)^{2L_{2}}\dots(C_{r}m+1)^{2L_{r}}}{\phi(C_{1}m+1)^{2L_{1}}\phi(C_{2}m+1)^{2L_{2}}\dots\phi(C_{r}m+1)^{2L_{r}}\phi(m)}\sum_{m \le t} \frac{m^{2L-2}}{\phi(m)^{2L-1}} \\
& \ll_{L_{1},\dots,L_{r},L,\gamma} \log^2 t
\end{align*} by Lemmas \ref{PhiSumLem} and \ref{Sieve2}.
\end{proof}

\section{More Summations involving $\pi(t,p,1)$}\label{MoreSums}

The previous section involved lemmas required to prove summations including terms such as $\pi(t,p,1).$ A lot of these summations will involve sieving techniques. This section will be split into proofs of two lemmas involving the summations required for the sums arising from the Propositions \ref{M1} and \ref{M2}.

\begin{lemma}\label{Sums}
Let $b,k,l$ be positive integers with $2 \le l \le k.$ Let $t>e^{e}$ be a real number and let constants $\al,\al_1,\al_2$  satisfy $0<\al<1/2$ and $0<\al_{1}<\al_{2}<1/2.$
\begin{enumerate}
\item[(a)] If $b>t^{\al},$ then 
\begin{equation}\label{Sum1}\sum_{p_{k} \in \p_{b}} \sum_{p_{k-1} \in \p_{p_k}}  \dots  \sum_{p_{2} \in \p_{p_3}} \pi(t;p_{2},1) \ll \frac{t \log t (\log \log t)^{k-2} }{b}.\end{equation}
\item[(b)] If $b\le t^{\al_{1}},$ then 
\begin{equation}\label{Sum2}\sum_{\substack{p_{l} \in \p_{b}\\p_ l>t^{\al_{2}}}} \sum_{p_{l-1} \in \p_{p_l}}  \dots \sum_{p_{2} \in \p_{p_3}} \pi(t;p_{2},1) \ll \frac{b^{l-1}t}{\phi(b)^{l} \log t}.\end{equation}
\item[(c)]  If $b\le t^{\al_{1}},$ then 
\begin{equation}\label{Sum3}\sum_{\substack{p_{l} \in \p_{b}}} \sum_{p_{l-1} \in \p_{p_l}}  \dots  \sum_{p_{2} \in \p_{p_3}} \pi(t;p_{2},1) \ll \frac{t (\log \log t)^{l-1}}{\phi(b) \log t}.\end{equation}

\end{enumerate} The implicit constants in $(a)-(c)$ depend on the choices of the $\alpha.$
\end{lemma}
\begin{proof}
For \eqref{Sum1} we just use the trivial estimate $\pi(t;p_{2},1) \le t/p_{2}$ and several uses of Brun-Titchmarsh \eqref{BT} to get
\begin{align*}
\sum_{p_{k} \in \p_{b}} \sum_{p_{k-1} \in \p_{k}}  \dots  \sum_{p_{2} \in \p_{3}} \pi(t;p_{2},1) & \le \sum_{p_{k} \in \p_{b}} \sum_{p_{k-1} \in \p_{k}} \dots  \sum_{p_{2} \in \p_{3}} \frac{t}{p_{2}} \\
& \ll t \sum_{p_{k} \in \p_{b}} \sum_{p_{k-1} \in \p_{k}}  \dots  \sum_{p_{3} \in \p_{4}} \frac{\log \log t}{p_{3}} \\
& \ll t \sum_{p_{k} \in \p_{b}} \frac{(\log \log t)^{k-2}}{p_{k}} \\
& \le t \sum_{\substack{m \equiv 1 \pmod{b} \\ t^{\alpha} \le m \le t}} \frac{(\log \log t)^{k-2}}{m} \\
& \le \frac{t \log t(\log \log t)^{k-2}}{b}
\end{align*} where $m>1$ and $m \equiv 1 \pmod{b}$ imply that $m>b$ and by using \eqref{Recip1}. As for \eqref{Sum2} we get
\begin{align*}
\sum_{\substack{p_{l} \in \p_{b}\\ l>t^{\al_{2}}}} & \sum_{p_{l-1} \in \p_{l}}  \dots  \sum_{p_{2} \in \p_{3}} \pi(t;p_{2},1) \\
& = \sum_{\substack{p_{l} \in \p_{b}\\ l>t^{\al_{2}}}} \sum_{p_{l-1} \in \p_{l}}  \dots  \sum_{p_{3} \in \p_{4}} \#\{ (m_{1},p_{2}): p_{2}=1 \pmod{p_{3}}, p_{2}>t^{\al_{2}},m_{1}p_{2}+1 \le t, p_{2}, m_{1}p_{2}+1 \text{ prime}  \} \\
& = \sum_{\substack{p_{l} \in \p_{b}\\ l>t^{\al_{2}}}} \sum_{p_{l-1} \in \p_{l}}  \dots  \sum_{p_{4} \in \p_{5}} \#\{ (m_{1},m_{2},p_{3}): p_{3}=1 \pmod{p_{4}}, p_{3}>t^{\al_{2}},m_{1}(m_{2}p_{3}+1)+1 \le t, \\ & \hskip10mm \{p_{3}, m_{2}p_{3}+1,   m_{1}(m_{2}p_{3}+1)+1\}  \text{ prime}  \} \\
&= \#\{ (m_{1},m_{2},\dots,m_{l-1},p_{l}): p_{l}=1 \pmod{b}, p_{l}>t^{\al_{2}}, m_{1}(m_{2}\dots(m_{l-2}(m_{l-1}p_{l}+1)+1)+\dots\\ & \hskip10mm +1 \le t, \{ p_{l}, m_{l-1}p_{l}+1, m_{l-2}(m_{l-1}p_{l}+1)+1,\dots,m_{1}(m_{2}\dots(m_{l-2}(m_{l-1}p_{l}+1)+1)\\& \hskip20mm+\dots+1\} \text{ prime}  \} \\
& \le \sum_{m_{1}\dotsm_{l-1} \le t^{1-\al_{2}}}\#\{ p_{l} < t/m_{1}\dots m_{l-1}: p_{l}=1 \pmod{b}, \\ &\hskip10mm \{ p_{l}, m_{l-1}p_{l}+1, m_{l-2}(m_{l-1}p_{l}+1)+1,\dots,m_{1}(m_{2}...(m_{l-2}(m_{l-1}p_{l}+1)+1)+\dots+1\} \text{ prime}  \}.
\end{align*}
From here will need to use Brun's Sieve method (see \cite[Theorem 2.4]{HR}) to get that

\begin{align*}
\#\{ p_{l} < & t/m_{1}\dots m_{l-1}: p_{l}=1 \pmod{b}, \\ &  \{ p_{l}, m_{l-1}p_{l}+1, m_{l-2}(m_{l-1}p_{l}+1)+1,\dots,m_{1}(m_{2}\dots(m_{l-2}(m_{l-1}p_{l}+1)+1)+\dots+1\} \text{ prime}  \} \\
& \ll \frac{E^{l-1}}{\phi(E)^{l-1}}\frac{b^{l-1}}{\phi(b)^{l-1}}\frac{bc_{1}\dots c_{l-1}}{\phi(bc_{1}\dots c_{l-1})}\frac{t/m_{1}\dots m_{l-1}b}{(\log t/m_{1}\dots m_{l-1}b)^{l}}
\end{align*}
where the $c_{i}$ and $E$ are 
\begin{align*}
E=&\bigg(\prod_{i=1}^{l-1}m_{i}^{i(i+1)/2}\bigg)(1+m_{1}+m_{1}m_{2}+\dots +m_{1}\dots m_{l-3})(1+m_{2}+m_{2}m_{3}+\dots +m_{2}\dots m_{l-3})\\ & \dots (1+m_{l-3})(1+m_{1}+m_{1}m_{2}+\dots +m_{1}\dots m_{l-4})(1+m_{2}+m_{2}m_{3}+\dots +m_{2}\dots m_{l-4})\\ & \dots (1+m_{l-4})\dots (1+m_1)
\end{align*}
and for $1 \le i \le l-1,$
\begin{align*}
c_{i}=1+ m_{i} + m_{i}m_{i+1} + \dots + m_{i} \dots m_{l-2}, c_{l-1}=1.
\end{align*}
Now using $\phi(mn) \le \phi(m)\phi(n)$ and $m_{1}\dots m_{l-1}b \le t^{1+\al_{1}-\al_{2}}$ where $1+\al_{1}-\al_{2} < 1$ we get

\begin{align*}
& \ll \frac{E^{l-1}}{\phi(E)^{l-1}}\frac{b^{l-1}}{\phi(b)^{l}}\frac{c_{1}}{\phi(c_{1})}\dots \frac{c_{l-1}}{\phi(c_{l-1})}\frac{t}{m_{1}\dots m_{l-1}(\log t)^{l}}.
\end{align*}
Using $$\frac{m^{L}}{\phi(m^{L})}=\frac{m}{\phi(m)},$$ we get the sum
$$\sum_{m_{1}\dots m_{l-1} \le t^{1-\al_{2}}}\frac{E^{l-1}}{\phi(E)^{l-1}}\frac{c_{1}}{\phi(c_{1})}\dots\frac{c_{l-1}}{\phi(c_{l-1})}\frac{1}{m_{1}\dots m_{l-1}} = \sum_{m_{1}\dots m_{l-1} \le t^{1-\al_{2}}}\frac{(E^*)^{l-1}}{\phi(E^*)^{l-1}}\frac{c_{1}}{\phi(c_{1})}\dots \frac{c_{l-1}}{\phi(c_{l-1})}\frac{1}{m_{1}\dots m_{l-1}}$$ where 

\begin{align*}
E^*=&(1+m_{1}+m_{1}m_{2}+\dots +m_{1}\dots m_{l-3})(1+m_{2}+m_{2}m_{3}+\dots +m_{2}\dots m_{l-3})\\ & \dots (1+m_{l-3})(1+m_{1}+m_{1}m_{2}+\dots +m_{1}\dots m_{l-4})(1+m_{2}+m_{2}m_{3}+\dots +m_{2}\dots m_{l-4})\\ & \dots (1+m_{l-4})\dots (1+m_1).
\end{align*}
We have that every factor in $E^*$ as well as the $c_{i}$ are of the form $1+Cm_{i}$ for some $i$ or of the form $m_{i}^L$. Hence using $l-1$ applications of Lemmas \ref{PhiSumLem}, \ref{Sieve2} or \ref{Sieve3} we can pick off the factors of the form $(1+Cm_i)$ one at a time.

\begin{align*}\sum_{m_{1}\dots m_{l-1} \le t^{1-\al_{2}}}& \frac{E^{l-1}}{\phi(E)^{l-1}}\frac{c_{1}}{\phi(c_{1})}...\frac{c_{l-1}}{\phi(c_{l-1})}\frac{1}{m_{1}\dots m_{l-1}} \\
& \ll \sum_{m_{2}\dots m_{l-1} \le t^{1-\al_{2}}}\frac{(E')^{l-1}}{\phi(E')^{l-1}}\frac{c'_{1}}{\phi(c'_{1})}\dots\frac{c'_{l-1}}{\phi(c'_{l-1})}\frac{1}{m_{2}\dots m_{l-1}}(\log t) \\
& \ll \sum_{m_{3}\dots m_{l-1} \le t^{1-\al_{2}}}\frac{(E'')^{l-1}}{\phi(E'')^{l-1}}\frac{c''_{1}}{\phi(c''_{1})}\dots \frac{c''_{l-1}}{\phi(c''_{l-1})}\frac{1}{m_{3}\dots m_{l-1}}(\log^2 t) \\
& \ll \dots  \ll (\log t)^{l-1}.\end{align*} where the $E^{(e)},c^{(e)}_i$ denote the $E^*$ and $c_i$ terms with the factors of the form $1+Cm_1$ through $1+Cm_e$ removed. Note that the $C$ are at most $1+t+t^2+\dots+t^{k-3}\le t^{k-2}$ and $l\le k$ so the implied constant only depends on $k$.
Therefore
$$\sum_{\substack{p_{l} \in \p_{b}\\ l>t^{\al_{2}}}} \sum_{p_{l-1} \in \p_{l}} \dots  \sum_{p_{2} \in \p_{3}} \pi(t;p_{2},1) \ll \frac{tb^{l-1}}{\phi(b)^{l}(\log t)^{l}}(\log t)^{l-1} = \frac{tb^{l-1}}{\phi(b)^{l}\log t}.$$
As for part (c), first note that $b/\phi(b) \ll \log\log b,$ so for $p_l>t^{\alpha_2}$, we get that part $(b)$ implies our bound. As for $p_l\le t^{\alpha_2}$ we'll split it into cases where $p_3$ is less than or greater than $t^{\alpha_2}.$ If $p_3\le t^{\alpha_2},$ then 

\begin{align*}
\sum_{\substack{p_{l} \in \p_{b}\\p_ l \le t^{\al_{2}}}} \sum_{p_{l-1} \in \p_{l}}  \dots  \sum_{\substack{p_{2} \in \p_{3}\\ p_2 \le t^{\alpha_2}}} \pi(t;p_{2},1) & \ll \sum_{\substack{p_{l} \in \p_{b}\\p_ l \le t^{\al_{2}}}} \sum_{p_{l-1} \in \p_{l}}  \dots  \sum_{\substack{p_{2} \in \p_{3}\\ p_2 \le t^{\alpha_2}}} \frac{t}{\phi(p_{2})\log t/p_{2}} \\
& \ll \sum_{\substack{p_{l} \in \p_{b}\\p_ l \le t^{\al_{2}}}} \sum_{p_{l-1} \in \p_{l}}\dots    \sum_{\substack{p_{2} \in \p_{3}\\ p_2 \le t^{\alpha_2}}} \frac{t}{p_{2}\log t} \\
& \ll \sum_{p_{l} \in \p_{b}}\frac{t (\log\log t)^{l-2}}{p_{l} \log t} \\
& \ll \frac{t (\log\log t)^{l-1}}{\phi(b) \log t}
\end{align*}
If $p_3> t^{\alpha_2},$ then since $b\le t^{\alpha_2}$ there is a minimum $m$ such that $p_m\le t^{\alpha_2}$. So using part (b) with $l=m$ we get

\begin{align*}
\sum_{\substack{p_{l} \in \p_{b}\\p_ l \le t^{\al_{2}}}} \sum_{p_{l-1} \in \p_{l}}  \dots  \sum_{\substack{p_{2} \in \p_{3}\\ p_2 > t^{\alpha_2}}} \pi(t;p_{2},1) & \ll \sum_{\substack{p_{l} \in \p_{b}\\p_ l \le t^{\al_{2}}}} \sum_{p_{l-1} \in \p_{l}}  \dots  \sum_{p_{m+1} \in \p_{m+2}} \frac{(p_{m-1})^{m-1}t}{\phi(p_{m-1})^m\log t} \\
& \ll \sum_{\substack{p_{l} \in \p_{b}\\p_ l \le t^{\al_{2}}}} \sum_{p_{l-1} \in \p_{l}}  \dots  \sum_{p_{m+1} \in \p_{m+2}} \frac{t}{p_{m-1}\log t}\\
& \ll \frac{t (\log\log t)^{l-m}}{\phi(b) \log t} \\
& \ll \frac{t (\log\log t)^{l-1}}{\phi(b) \log t}
\end{align*}
since $m\ge 2$ and by using Brun-Titchmarsh \eqref{BT} which finishes part (c) and the lemma.
\end{proof}
As for the summations requires for the second moment, we'll note that we need twice as many sums due to $h_k(p)^2$. However the techniques required are similar.

\begin{lemma}\label{Doubles} Let $t> e^e$ and $0<2\alpha_1 < \alpha_2 <1/2$. Then
\begin{enumerate}
\item[(a)] If $b_{1}>t^{\alpha_{1}}$ or $b_{2}>t^{\alpha_{1}}$ then
\begin{equation} \label{Doublea}
\sum_{\substack{p_{2} \in \p_{b_{1}} \\ r_{2} \in \p_{b_{2}}}} \pi(t;p_{2}r_{2},1) \ll \frac{t\log^{2}t}{b_{1}b_{2}}.
\end{equation}
\item[(b)] If neither $b_{1}$ nor $b_{2}$ exceeds $t^{\alpha_{1}},$ then 
\begin{equation} \label{Doubleb}
\sum_{\substack{p_{k} \in \p_{b_{1}} \\ r_{k} \in \p_{b_{2}} \\ p_{k}r_{k}>t^{\alpha_2} }} ... \sum_{\substack{p_{2} \in \p_{p_{3}} \\ r_{2} \in \p_{r_{3}}}} \pi(t;p_{2}r_{2},1) \ll \frac{t(\log \log t)^{k-1}b_{2}^{k-1}}{\phi(b_{1})\phi(b_{2})^{k}\log t} + \frac{t(\log \log t)^{k-1}b_{1}^{k-1}}{\phi(b_{2})\phi(b_{1})^{k}\log t}.
\end{equation}
\item[(c)] If neither $b_{1}$ nor $b_{2}$ exceeds $t^{\alpha_{1}},$ then 
\begin{equation} \label{Doublec}
\sum_{\substack{p_{k} \in \p_{b_{1}} \\ r_{k} \in \p_{b_{2}}}} ... \sum_{\substack{p_{2} \in \p_{p_{3}} \\ r_{2} \in \p_{r_{3}}}} \pi(t;p_{2}r_{2},1) \ll \frac{t (\log\log t)^{2k-2}}{\phi(b_{1})\phi(b_{2}) \log t}.
\end{equation}
\item[(d)] If neither $b_{1}$ nor $b_{2}$ exceeds $t^{\alpha_{1}},$ then 
\begin{equation} \label{Doubled}
\sum_{\substack{p_{k} \in \p_{b_{1}} \\ r_{k} \in \p_{b_{2}}}} ... \sum_{\substack{p_{3} \in \p_{p_{4}} \\ r_{3} \in \p_{r_{4}}}}\sum_{s \in \p_{p_{3}} \cap \p_{r_{3}}} \pi(t;s,1) \ll \frac{t (\log\log t)^{2k-2}}{\phi(b_{1})\phi(b_{2}) \log t}.
\end{equation}
\end{enumerate} Again the implicit constants depend on our choice of the $\alpha.$
\end{lemma}
\begin{proof}

(a) is similar to part (a) of Lemma \ref{Sums}.
For part (b) we first assume that $p_{k} \le r_{k}$, then
\begin{align*}
\sum_{\substack{p_{k} \in \p_{b_{1}} \\ r_{k} \in \p_{b_{2}}\\ p_{k}\le r_{k} \\ p_{k}r_{k}>t^{\alpha_2} }} & ... \sum_{\substack{p_{2} \in \p_{p_{3}} \\ r_{2} \in \p_{r_{3}}}} \pi(t;p_{2}r_{2},1) \\
& = \sum_{\substack{p_{k} \in \p_{b_{1}} \\ r_{k} \in \p_{b_{2}} \\ p_{k}r_{k}>t^{\alpha_2} }} \dots \sum_{\substack{p_{3} \in \p_{p_{4}} \\ r_{3} \in \p_{r_{4}}}} \#\{ (m_{1},p_{2},r_{2}): p_{2}=1 \pmod{p_{3}}, r_{2}=1 \pmod{r_{3}}, r_{2}p_{2}>t^{\al_{2}}, \\ & \hskip10mm m_{1}r_{2}p_{2}+1 \le t, p_{2}, m_{1}r_{2}p_{2}+1 \text{ prime}  \} \\
& = \sum_{\substack{p_{k} \in \p_{b_{1}} \\ p_{k}\le r_{k}}} \sum_{p_{k-1} \in \p_{k}}  \dots  \sum_{p_{2} \in \p_{3}}\sum_{\substack{r_{k} \in \p_{b_{2}}\\ p_{k}r_{k}>t^{\al_{2}}}} \sum_{r_{k-1} \in \p_{r_{k}}}  \dots \sum_{r_{4} \in \p_{r_{5}}} \#\{ (m_{1},m_{2},r_{3}): r_{3}=1 \pmod{r_{4}}, \\ & \hskip10mm r_{3}p_{2}>t^{\al_{2}},m_{1}p_{2}(m_{2}r_{3}+1)+1 \le t, \{r_{3}, m_{2}r_{3}+1,   m_{1}p_{2}(m_{2}r_{3}+1)+1\}  \text{ prime}  \} \\
&= \sum_{\substack{p_{k} \in \p_{b_{1}} \\ p_{k}\le r_{k}}} \sum_{p_{k-1} \in \p_{k}}  \dots  \sum_{p_{2} \in \p_{3}} \#\{ (m_{1},m_{2},\dots,m_{l-1},r_{l}): r_{l}=1 \pmod{b_{2}}, p_{2}r_{k}>t^{\al_{2}}, \\ & \hskip10mm m_{1}p_{2}(m_{2}\dots (m_{k-2}(m_{k-1}r_{k}+1)+1)+\dots+1 \le t, \{ r_{k}, m_{k-1}r_{k}+1, \\ & \hskip20mm m_{k-2}(m_{k-1}r_{k}+1)+1, \dots, \\ & \hskip30mm m_{1}p_{2}(m_{2}\dots (m_{k-2}(m_{k-1}r_{k}+1)+1)+\dots+1\} \text{ prime}  \} \\
& \le \sum_{m_{1}\dots m_{l-1} \le t^{1-\al_{2}}}\sum_{\substack{p_{k} \in \p_{b_{1}} \\ p_{k}\le r_{k}}} \sum_{p_{k-1} \in \p_{k}}\dots  \sum_{p_{2} \in \p_{3}} \#\{ r_{k} < t/p_{2}m_{1}...m_{k-1}: r_{k}=1 \pmod{b_{2}},  \\ & \hskip10mm \{ r_{k}, m_{k-1}r_{k}+1, m_{k-2}(m_{k-1}r_{k}+1)+1,\dots,\\ & \hskip20mm p_{2}m_{1}(m_{2}\dots (m_{k-2}(m_{k-1}r_{k}+1)+1)+\dots+1\} \text{ prime}  \}
\end{align*}
Just like in Lemma \ref{Sums} we use Brun's Sieve. However, notice that we have almost the same set, except with $m_{1}$ replaced with $m_{1}p_{2}.$ Hence we have

\begin{align*}
\#\{ r_{k} & < t/p_{2}m_{1}\dotsm_{k-1}: r_{k}=1 \pmod{b_{1}}, \{ r_{k}, m_{k-1}r_{k}+1, m_{k-2}(m_{k-1}r_{k}+1)+1,  \\ & \hskip10mm \dots,p_{2}m_{1}(m_{2}\dots (m_{k-2}(m_{k-1}r_{k}+1)+1)+\dots+1\} \text{ prime}  \} \\
& \ll \frac{E^{k-1}}{\phi(E)^{k-1}}\frac{b_{2}^{k-1}}{\phi(b_{2})^{k-1}}\frac{b_{2}c_{1}\dots c_{k-1}}{\phi(b_{2}c_{1}\dots c_{k-1})}\frac{t/p_{2}m_{1}\dots m_{k-1}b_{2}}{(\log t/p_{2}m_{1}\dots m_{k-1}b_{2})^{k}}
\end{align*}
where the $c_{i}$ and $E$ are 
\begin{align*}
E=& p_{2}\bigg(\prod_{i=1}^{l-1}m_{i}^{i(i+1)/2}\bigg)(1+p_{2}m_{1}+p_{2}m_{1}m_{2}+...+p_{2}m_{1}\dots m_{k-3})(1+m_{2}+m_{2}m_{3}+\dots\\ & \hskip10mm +m_{2}\dots m_{k-3}) \dots (1+m_{k-3})(1+p_{2}m_{1}+p_{2}m_{1}m_{2}+\dots+p_{2}m_{1}\dots m_{k-4})\\ & \hskip20mm (1+m_{2}+m_{2}m_{3}+\dots+m_{2}\dots m_{k-4}) \dots (1+m_{k-4})\dots (1+p_{2}m_1)
\end{align*}
and for $2 \le i \le k-1,$
\begin{align*}
& c_{1}=1+ p_{2}m_{1} + p_{2}m_{1}m_{2} + \dots + p_{2}m_{1}\dots m_{k-2},\\ & \hskip10mm c_{i}=1+ m_{i} + m_{i}m_{i+1} + \dots + m_{i}\dots m_{k-2}, c_{k-1}=1.
\end{align*}
By the same methods as Lemma \ref{Sums} and using that $p_{2}/\phi(p_{2})$ is bounded and noting that 
$$\frac{t}{p_{2}m_{1}...m_{k-1}b_{2}} > \frac{r_{k}}{b_{1}} > t^{\alpha_{2}/2-\alpha_{1}}=t^{\epsilon}$$ for some $\epsilon>0$ since $\alpha_{2}>2\alpha_{1}$, we get that

\begin{align*}
\sum_{\substack{p_{k} \in \p_{b_{1}} \\ r_{k} \in \p_{b_{2}} \\ p_{k}r_{k}>t^{\alpha_2} }} \dots \sum_{\substack{p_{2} \in \p_{p_{3}} \\ r_{2} \in \p_{r_{3}}}} \pi(t;p_{2}r_{2},1) & \ll \frac{tb_{2}^{k-1}}{\phi(b_{2})^{k}\log t}\sum_{p_{k} \in \p_{b_{1}}} \sum_{p_{k-1} \in \p_{k}}  \dots  \sum_{p_{2} \in \p_{3}}\frac{1}{p_{2}} \\
& \ll \frac{tb_{2}^{k-1}}{\phi(b_{2})^{k}\log t}\sum_{p_{k} \in \p_{b_{1}}}\frac{(\log \log t)^{k-2}}{p_{k}} \\
& \ll \frac{t(\log \log t)^{k-1}b_{2}^{k-1}}{\phi(b_{1})\phi(b_{2})^{k}\log t}.
\end{align*} The case for $r_{k}\le p_{k}$ is similar.
As for part (c), first note that $b_{i}/\phi(b_{i}) \ll \log\log b_{i}$ for $i \in \{1,2\}.$ taking care of the case where $p_{k}r_{k}>t^{\alpha_{2}}.$ As for $p_{k}r_{k}\le t^{\alpha_{2}}$ we get
\begin{align*}
\sum_{\substack{p_{k} \in \p_{b_{1}} \\ r_{k} \in \p_{b_{2}} \\ p_{k}r_{k}\le t^{\alpha_2} }} \dots \sum_{\substack{p_{2} \in \p_{p_{3}} \\ r_{2} \in \p_{r_{3}}}} \pi(t;p_{2}r_{2},1) & \ll \sum_{\substack{p_{k} \in \p_{b_{1}} \\ r_{k} \in \p_{b_{2}} \\ p_{k}r_{k}\le t^{\alpha_2} }} \dots \sum_{\substack{p_{2} \in \p_{p_{3}} \\ r_{2} \in \p_{r_{3}}}} \frac{t}{\phi(p_{2}r_{2})\log t/p_{2}r_{2}} \\
& \ll \sum_{\substack{p_{k} \in \p_{b_{1}} \\ r_{k} \in \p_{b_{2}} \\ p_{k}r_{k}\le t^{\alpha_2} }} \dots \sum_{\substack{p_{2} \in \p_{p_{3}} \\ r_{2} \in \p_{r_{3}}}} \frac{t}{p_{2}r_{2}\log t} \\
& \ll \sum_{\substack{p_{k} \in \p_{b_{1}} \\ r_{k} \in \p_{b_{2}} \\ p_{k}r_{k}\le t^{\alpha_2} }}\frac{t(\log \log t)^{2k-4}}{p_{k}r_{k}\log t} \\ \\
& \ll \frac{t (\log\log t)^{2k-2}}{\phi(b_{1})\phi(b_{2}) \log t}
\end{align*} using Brun-Titchmarsh, \eqref{BT} finishing part (c). As for part (d) we note that
\begin{align*}
\sum_{\substack{p_{3} \in \p_{p_{4}} \\ r_{3} \in \p_{r_{4}}}}& \sum_{s \in \p_{p_{3}} \cap \p_{r_{3}}} \pi(t;s,1) \\
& = \sum_{\substack{p_{3} \in \p_{p_{4}} \\ r_{3} \in \p_{r_{4}}}}\#\{ (m_{1},s): s=1 \pmod{p_{3}r_{3}}, m_{1}s+1 \le t, s, m_{1}s+1 \text{ prime}  \} \\
& = \sum_{p_{3} \in \p_{p_{4}}} \#\{ (m_{1},m_{2},r_{3}): r_{3}=1 \pmod{r_{4}}, m_{1}(m_{2}p_{3}r_{3}+1)+1 \le t, \\ & \hskip30mm \{m_{2}p_{3}r_{3}+1, m_{1}(m_{2}p_{3}r_{3}+1)+1 \text{ prime}  \} \\
\end{align*} and so on, yielding a similar sieve as part (b). 
\end{proof}

\section{Reduction of $\sum h_k(p)$ to small values of $p_k$}

We will be using Euler Summation on the sum $\sum_{p \leq t}\hk (p)$ in our efforts to find our estimate for $M_1(x).$ It will turn out that the large primes do not contribute much to the some. The sum will involve estimating $\pi(t;p,1)$ by $\li(t)/p-1.$ The following lemma will deal with those errors and will involve the Bombieri--Vinogradov Theorem.

\begin{lemma}\label{BVreduction}
For all $2\le l\le k$, $x>e^{e^{e}}$ and $v>e^{e},$
\begin{align*}\sum_{q \leq y^{k}} \log q & \sum_{a \in \nat} \sum_{\substack{p_{k} \in \p_{q^{a}} \\ p_{k} \leq v^{1/3^{l-1}}}} \sum_{\substack{p_{k-1} \in \p_{p_{k}}\\ p_{k-1} \leq v^{1/3^{l-2}}}}  \dots \sum_{\substack{p_{k-l+2} \in \p_{p_{k-l+3}}\\ p_{k-l+2} \leq v^{1/3}}}\bigg(\pi(v,p_{k-l+2},1)- \frac{\li(v)}{p_{k-l+2}}\bigg) \\
& \ll \frac{v\log y}{\log v}+\li(v)(\log\log v)^{l-2}.\end{align*}
\end{lemma}
\begin{proof}
Let $E(t;r,1)=\pi(t;r,1)-\frac{\li(t)}{r-1}.$ Then we have
\begin{align*}\sum_{q \leq y^{k}} \log q & \sum_{a \in \nat} \sum_{\substack{p_{k} \in \p_{q^{a}} \\ p_{k} \leq v^{1/3^{l-1}}}} \sum_{\substack{p_{k-1} \in \p_{p_{k}}\\ p_{k-1} \leq v^{1/3^{l-2}}}}  \dots \sum_{\substack{p_{k-l+2} \in \p_{p_{k-l+3}}\\ p_{k-l+2} \leq v^{1/3}}}\bigg(\pi(v,p_{k-l+2},1)- \frac{\li(v)}{p_{k-l+2}-1}\bigg)\\
& = \sum_{q \leq y^{k}} \log q \sum_{a \in \nat} \sum_{\substack{p_{k} \in \p_{q^{a}} \\ p_{k} \leq v^{1/3^{l-1}}}} \sum_{\substack{p_{k-1} \in \p_{p_{k}}\\ p_{k-1} \leq v^{1/3^{l-2}}}}  \dots \sum_{\substack{p_{k-l+2} \in \p_{p_{k-l+3}}\\ p_{k-l+2} \leq v^{1/3}}} E(v; p_{k-l+2},1) \\
& \ll \sum_{q \leq y^{k}} \log q \sum_{a \in \nat} \sum_{\substack{p_{k} \in \p_{q^{a}} \\ p_{k} \leq v^{1/3^{l-1}}}} \sum_{\substack{p_{k-1} \in \p_{p_{k}}\\ p_{k-1} \leq v^{1/3^{l-2}}}}  \dots \sum_{\substack{p_{k-l+2} \in \p_{p_{k-l+3}}\\ p_{k-l+2} \leq v^{1/3}}} \abs{E(v; p_{k-l+2},1)}.\end{align*}
Let $\Omega(m)$ denote the number of divisors of $m$ which are primes or prime powers. We use the estimate $\Omega(m) \ll \log m$ to get

\begin{align*}\sum_{q \leq y^{k}}& \log q \sum_{a \in \nat} \sum_{\substack{p_{k} \in \p_{q^{a}} \\ p_{k} \leq v^{1/3^{l-1}}}} \sum_{\substack{p_{k-1} \in \p_{p_{k}}\\ p_{k-1} \leq v^{1/3^{l-2}}}}  \dots \sum_{\substack{p_{k-l+2} \in \p_{p_{k-l+3}}\\ p_{k-l+2} \leq v^{1/3}}} \abs{E(v; p_{k-l+2},1)} \\
& \leq \log(y^{k})\sum_{\substack{p_{k-l+2} \in \p_{p_{k-l+3}}\\ p_{k-l+2} \leq v^{1/3}}} \abs{E(v; p_{k-l+2},1)} \sum_{\substack{p_{k-l+3} \mid p_{k-l+2}-1 \\ p_{3} \leq v^{1/9}}} \sum_{\substack{p_{k-l+4} \mid p_{k-l+3}-1 \\ p_{k-l+4} \leq v^{1/27}}} \dots \sum_{q \leq y^{k}} \sum_{\substack{a \in \nat \\ q^{a} \mid p_{k}-1}}1  \\
& \leq \log(y^{k})\sum_{\substack{p_{k-l+2} \in \p_{p_{k-l+3}}\\ p_{k-l+2} \leq v^{1/3}}} \abs{E(v; p_{k-l+2},1)} \sum_{\substack{p_{k-l+3} \mid p_{k-l+2}-1 \\ p_{3} \leq v^{1/9}}} \sum_{\substack{p_{k-l+4} \mid p_{k-l+3}-1 \\ p_{k-l+4} \leq v^{1/27}}} \dots \sum_{\substack{p_{k}\le v^{1/3^{k-1}} \\ p_{k} \mid p_{k-1}-1}}\Omega(p_{k}-1)  \\
& \ll \log y\sum_{\substack{p_{k-l+2} \in \p_{p_{k-l+3}}\\ p_{k-l+2} \leq v^{1/3}}} \abs{E(v; p_{k-l+2},1)} \sum_{\substack{p_{k-l+3} \mid p_{k-l+2}-1 \\ p_{3} \leq v^{1/9}}} \sum_{\substack{p_{k-l+4} \mid p_{k-l+3}-1 \\ p_{k-l+4} \leq v^{1/27}}} \dots \sum_{\substack{p_{k}\le v^{1/3^{k-1}} \\ p_{k} \mid p_{k-1}-1}}\log t.\end{align*} Continuing in this manner we obtain
\begin{align*}
\sum_{q \leq y^{k}} \log q \sum_{a \in \nat}& \sum_{\substack{p_{k} \in \p_{q^{a}} \\ p_{k} \leq v^{1/3^{l-1}}}} \sum_{\substack{p_{k-1} \in \p_{p_{k}}\\ p_{k-1} \leq v^{1/3^{l-2}}}}  \dots \sum_{\substack{p_{k-l+2} \in \p_{p_{k-l+3}}\\ p_{k-l+2} \leq v^{1/3}}} \abs{E(v; p_{k-l+2},1)}\\ & \ll \log y (\log v)^{l-1} \sum_{\substack{p_{k-l+2} \in \p_{p_{k-l+3}}\\ p_{k-l+2} \leq v^{1/3}}} \abs{E(v; p_{k-l+2},1)} \ll \frac{v \log y}{\log t}\end{align*} using Bombieri--Vinogradov. As for the difference between 
$$\sum_{q \leq y^{k}} \log q \sum_{a \in \nat} \sum_{\substack{p_{k} \in \p_{q^{a}} \\ p_{k} \leq v^{1/3^{l-1}}}} \sum_{\substack{p_{k-1} \in \p_{p_{k}}\\ p_{k-1} \leq v^{1/3^{l-2}}}}  \dots \sum_{\substack{p_{k-l+2} \in \p_{p_{k-l+3}}\\ p_{k-l+2} \leq v^{1/3}}}\frac{\li(v)}{p_{k-l+2}-1}$$ and
\begin{equation}\label{Miles}\sum_{q \leq y^{k}} \log q \sum_{a \in \nat} \sum_{\substack{p_{k} \in \p_{q^{a}} \\ p_{k} \leq v^{1/3^{l-1}}}} \sum_{\substack{p_{k-1} \in \p_{p_{k}}\\ p_{k-1} \leq v^{1/3^{l-2}}}}  \dots \sum_{\substack{p_{k-l+2} \in \p_{p_{k-l+3}}\\ p_{k-l+2} \leq v^{1/3}}}\frac{\li(v)}{p_{k-l+2}}\end{equation} we get that it is
\begin{align*}\sum_{q \leq y^{k}} \log q & \sum_{a \in \nat} \sum_{\substack{p_{k} \in \p_{q^{a}} \\ p_{k} \leq v^{1/3^{l-1}}}} \sum_{\substack{p_{k-1} \in \p_{p_{k}}\\ p_{k-1} \leq v^{1/3^{l-2}}}}  \dots \sum_{\substack{p_{k-l+2} \in \p_{p_{k-l+3}}\\ p_{k-l+2} \leq v^{1/3}}} \frac{\li(v)}{p_{k-l+2}(p_{k-l+2}-1)} \\
& \leq \sum_{q \leq y^{k}} \log q \sum_{a \in \nat} \sum_{\substack{p_{k} \in \p_{q^{a}} \\ p_{k} \leq v^{1/3^{l-1}}}} \sum_{\substack{p_{k-1} \in \p_{p_{k}}\\ p_{k-1} \leq v^{1/3^{l-2}}}}  \dots  \sum_{i=1}^{\infty} \frac{\li(v)}{(ip_{k-l+3}+1)(ip_{k-l+3})} \\
&\ll \sum_{q \leq y^{k}} \log q \sum_{a \in \nat} \sum_{\substack{p_{k} \in \p_{q^{a}} \\ p_{k} \leq v^{1/3^{l-1}}}} \sum_{\substack{p_{k-1} \in \p_{p_{k}}\\ p_{k-1} \leq v^{1/3^{l-2}}}}  \dots \sum_{\substack{p_{k-l+3} \in \p_{p_{k-l+4}}\\ p_{k-l+3} \leq v^{1/9}}} \frac{\li(v)}{p_{k-l+3}^{2}} \\
&\ll \sum_{q \leq y^{k}} \log q \sum_{a \in \nat} \sum_{\substack{p_{k} \in \p_{q^{a}} \\ p_{k} \leq v^{1/3^{l-1}}}} \sum_{\substack{p_{k-1} \in \p_{p_{k}}\\ p_{k-1} \leq v^{1/3^{l-2}}}}  \dots \sum_{\substack{p_{k-l+3} \in \p_{p_{k-l+4}}\\ p_{k-l+3} \leq v^{1/9}}} \frac{\li(v)}{p_{k-l+3}q^a} \\
& \ll \sum_{q \leq y^{k}} \log q \sum_{a \in \nat} \frac{\li(v)(\log\log v)^{l-2}}{q^{2a}} \\
& \ll \sum_{q \leq y^{k}}\frac{\li(v)(\log\log v)^{l-2} \log q}{q^{2}} \\ 
& \ll \li(v)(\log\log v)^{l-2}\end{align*} using the Brun--Titchmarsh inequality \eqref{BT}, the inequality $p_{k-l+3} \ge q^a$ and noting that the sum over $q$ converges.
\end{proof}

\begin{lemma}\label{hk1} For all $x>e^{e^{e}}$ and $t>e^{e}$,
\begin{align*}\sum_{p \leq t}\hk (p) = \sum_{q \leq y^{k}} \log q & \sum_{a \in \nat} \sum_{\substack{p_{k} \in \p_{q^{a}} \\ p_{k} \leq t^{1/3^{k-1}}}} \sum_{\substack{p_{k-1} \in \p_{p_{k}}\\ p_{k-1} \leq t^{1/3^{k-2}}}}  \dots \sum_{\substack{p_{2} \in \p_{p_{3}}\\ p_{2} \leq t^{1/3}}} \pi(t;p_2,1) \\ & + O\bigg(t^{1-1/3^{k}} \log t (\log \log t)^{k-2} y^{k}+\frac{t (\log \log t)^{k-2}\log y}{\log t} \bigg).\end{align*}
\end{lemma}
\begin{proof}
For a prime $p$, 
\begin{align*}\hk (p) &= \sum_{p_{1}\mid p} \sum_{p_{2}\mid p_{1}-1} \dots \sum_{p_{k}\mid p_{k-1}-1} \sum_{q \leq y^{k}} \nu_{q}(p_{k}-1)\log q \\
& =\sum_{p_2 \mid p-1} \dots \sum_{p_{k}\mid p_{k-1}-1} \sum_{q \leq y^{k}} \nu_{q}(p_{k}-1)\log q\end{align*} since the only prime which can divide $p$ is $p$ itself. Hence

\begin{align*}\sum_{p \leq t}\hk (p) &= \sum_{p \leq t} \sum_{p_{2}\mid p-1} \dots \sum_{p_{k}\mid p_{k-1}-1} \sum_{q \leq y^{k}} \nu_{q}(p_{k}-1)\log q \\
& =  \sum_{p \leq t} \sum_{p_{2}\mid p_{1}-1} \dots \sum_{p_{k}\mid p_{k-1}-1} \sum_{q \leq y^{k}} \sum_{\substack{p_{k} \in \p_{q^{a}} \\ a \in \nat}} \log q \\
& = \sum_{q \leq y^{k}} \log q \sum_{a \in \nat} \sum_{p_{k} \in \p_{q^{a}}} \sum_{p_{k-1} \in \p_{p_{k}}}  \dots  \sum_{p_2 \in \p_{p_{3}}} \sum_{\substack{p \leq t\\ p \in \p_{p_{2}}}} 1 \\
& = \sum_{q \leq y^{k}} \log q \sum_{a \in \nat} \sum_{p_{k} \in \p_{q^{a}}} \sum_{p_{k-1} \in \p_{p_{k}}}  \dots  \sum_{p_2 \in \p_{p_{3}}} \pi(t;p_{2},1).\end{align*}

We wish to approximate $\pi(t;p_{2},1)$ by $\frac{\li(t)}{p_{2}-1}$ and use the Bombieri-Vinogradov Theorem to deal with the error. However this approximation only allows primes up to say $t^{1/3}.$ So we use the estimations in Lemma \ref{Sums} to bound these errors. We will see that the main contribution comes from $p_{i} \leq t^{1/3^{i-1}}$ and $q^{a} \leq t^{1/3^{k}}.$

Using Lemma \ref{Sums}, we get for large $q^{a}$

$$\sum_{q \leq y^{k}} \log q \sum_{\substack{a \in \nat \\ q^{a} > t^{1/3^{k}}}} \sum_{p_{k} \in \p_{q^{a}}} \sum_{p_{k-1} \in \p_{p_{k}}}  \dots  \sum_{p_{3} \in \p_{p_{2}}} \pi(t;p_{2},1)
\ll \sum_{q \leq y^{k}} \log q \sum_{\substack{a \in \nat \\ q^{a} > t^{1/3^{k}}}} \frac{t \log t (\log \log t)^{k-2} }{q^{a}}.$$
By geometric estimates, if $a^{*}$ is the smallest $a$ where  $q^{a} > t^{1/3^{k}}$, then we get that the above is

\begin{align*}&\ll t \log t (\log \log t)^{k-2} \sum_{q \leq y^{k}} \frac{\log q }{q^{a^{*}}} \\
& \leq  t^{1-1/3^{k}} \log t (\log \log t)^{k-2}  \sum_{q \leq y^{k}} \log q \\
& \ll t^{1-1/3^{k}} \log t (\log \log t)^{k-2} y^{k}.\end{align*} Now suppose $q^{a} \leq t^{1/3^{k}}.$ Let $l$ be the last index (supposing one exists) where $p_{i} > t^{1/3^{i-1}}$ By using \eqref{Sum2} where $l$ ranges from $2$ to $k,$ we can bound the large values of the $p_i$.

\begin{align*}\sum_{q \leq y^{k}} \log q & \sum_{\substack{a \in \nat \\ q^{a} \leq t^{1/3^{k}}}} \sum_{\substack{p_{k} \in \p_{q^{a}} \\ p^{k} \leq t^{1/3^{k-1}}}}\dots \sum_{\substack{p_{l+1} \in \p_{p_{l+2}} \\ p_{l+1} \leq t^{1/3^{l}}}} \sum_{\substack{p_{l} \in \p_{p_{l+1}} \\ p_{l} > t^{1/3^{l-1}}}} \sum_{p_{l-1} \in \p_{p_{l}}}  \dots  \sum_{p_{2} \in \p_{p_{3}}} \pi(t;p_{2},1) \\
& \ll \sum_{q \leq y^{k}} \log q \sum_{\substack{a \in \nat \\ q^{a} \leq t^{1/3^{k}}}} \sum_{\substack{p_{k} \in \p_{q^{a}} \\ p_{k} \leq t^{1/3^{k-1}}}}\dots\sum_{\substack{p_{l+2} \in \p_{p_{l+3}} \\ p_{l+2} \leq t^{1/3^{l+1}}}} \sum_{\substack{p_{l+1} \in \p_{p_{l+2}} \\ p_{l+1} > t^{1/3^{l}}}} \frac{(p_l)^{l-1}t }{\phi(p_{l})^{l} \log t} \\ & \ll \sum_{q \leq y^{k}} \log q \sum_{\substack{a \in \nat \\ q^{a} \leq t^{1/3^{k}}}} \sum_{\substack{p_{k} \in \p_{q^{a}} \\ p_{k} \leq t^{1/3^{k-1}}}}\dots \sum_{\substack{p_{l+2} \in \p_{p_{l+3}} \\ p_{l+2} \leq t^{1/3^{l+1}}}} \sum_{\substack{p_{l+1} \in \p_{p_{l+2}} \\ p_{l+1} > t^{1/3^{l}}}} \frac{t }{p_{l+1} \log t}\end{align*} since $p_l$ is prime and $l\le k$. By Brun-Titchmarsh \eqref{BT} we get 
\begin{align*}& \ll \sum_{q \leq y^{k}} \log q \sum_{\substack{a \in \nat \\ q^{a} \leq t^{1/3^{k}}}}  \frac{t (\log \log t)^{k-l}}{q^{a} \log t} \\
& \ll \sum\limits_{q \leq y^{k}} \frac{t (\log \log t)^{k-l}\log q}{q \log t} \\
& \ll \frac{t (\log \log t)^{k-2}\log y}{\log t}\end{align*} by \eqref{Merten} and since $l\ge 2$. Hence we get
\begin{align*}
\sum_{p \leq t}\hk (p) = \sum_{q \leq y^{k}} \log q & \sum_{\substack{a \in \nat \\ q^{a} \leq t^{1/3^{k}}}} \sum_{\substack{p_{k} \in \p_{q^{a}} \\ p_{k} \leq t^{1/3^{k-1}}}} \sum_{\substack{p_{k-1} \in \p_{p_{k}}\\ p_{k-1} \leq t^{1/3^{k-2}}}}  \dots  \sum_{\substack{p_{2} \in \p_{p_{3}}\\ p_{2} \leq t^{1/3}}} \pi(t,p_{2},1) \\
& + O\bigg(t^{1-1/3^{k}} \log t (\log \log t)^{k-2} y^{k}+\frac{t (\log \log t)^{k-2}\log y}{\log t} \bigg)\end{align*} finishing the lemma.
\end{proof}

\section{Evaluation of the Main Term}\label{EvalMainTerm}
Now we'll deal with the main term from Lemma \ref{hk1}. We will deal with estimating the individual sums recursively. Hence we wish to make the following definition.

\begin{definition}Let $2 \le l\le k$ and $2\le u \le t$. Then define
$$g_{k,l}(u)=\sum_{q \leq y^{k}} \log q \sum_{a \in \nat} \sum_{\substack{p_{k} \in \p_{q^{a}} \\ p_{k} \leq u^{1/3^{l-1}}}} \sum_{\substack{p_{k-1} \in \p_{p_{k}}\\ p_{k-1} \leq u^{1/3^{l-2}}}}  \dots  \sum_{\substack{p_{k-l+2} \in \p_{p_{k-l+3}}\\ p_{k-l+2} \leq u^{1/3}}} \pi(u;p_{k-l+2},1).$$
\end{definition}
Note that $g_{k,k}(t)$ is the summation in Lemma \ref{hk1}. Next we'll exhibit the recursive formula satisfied by the $g_{k,l}$.

\begin{lemma}\label{Recursion}
Let $3\le l\le k$, then
\begin{equation}
g_{k,l}(v)=\li(v)\int_{2}^{v^{1/3}} \frac{1}{u^{2}} g_{k,l-1}(u)du + O\bigg(\frac{v (\log \log v)^{l-2}\log y}{\log v} \bigg).
\end{equation}
\end{lemma}
\begin{proof}
We'll proceed by approximating $\pi$ by $\li$ and then use partial summation to recover $\pi.$ Using Lemma \ref{BVreduction} we get
\begin{align*}
g_{k,l}(v)&= \sum_{q \leq y^{k}} \log q \sum_{a \in \nat} \sum_{\substack{p_{k} \in \p_{q^{a}} \\ p_{k} \leq v^{1/3^{l-1}}}} \sum_{\substack{p_{k-1} \in \p_{p_{k}}\\ p_{k-1} \leq v^{1/3^{l-2}}}}  \dots  \sum_{\substack{p_{k-l+2} \in \p_{p_{k-l+3}}\\ p_{k-l+2} \leq v^{1/3}}} \pi(v;p_{k-l+2},1) \\ & \ll \sum_{q \leq y^{k}} \log q \sum_{a \in \nat} \sum_{\substack{p_{k} \in \p_{q^{a}} \\ p_{k} \leq v^{1/3^{l-1}}}} \sum_{\substack{p_{k-1} \in \p_{p_{k}}\\ p_{k-1} \leq v^{1/3^{l-2}}}}  \dots  \sum_{\substack{p_{k-l+2} \in \p_{p_{k-l+3}}\\ p_{k-l+2} \leq v^{1/3}}} \frac{\li(v)}{p_{k-l+2}}+O\bigg(\frac{v\log y}{\log v}+ \li(v)(\log\log v)^{l-2} \bigg).
\end{align*}
We use Euler summation on the inner sum to get
$$\sum_{\substack{p_{k-l+2} \in \p_{p_{k-l+3}}\\ p_{k-l+2} \leq v^{1/3}}} \frac{1}{p_{k-l+2}} = \frac{\pi(v^{1/3};p_{k-l+3},1)}{v^{1/3}}+\int_{2}^{v^{1/3}} \frac{\pi(u;p_{k-l+3},1)}{u^{2}}du $$ and so we get that 
\begin{align*}
g_{k,l}(v)&=\li(v)\sum_{q \leq y^{k}} \log q \sum_{a \in \nat} \sum_{\substack{p_{k} \in \p_{q^{a}} \\ p_{k} \leq v^{1/3^{l-1}}}} \sum_{\substack{p_{k-1} \in \p_{p_{k}}\\ p_{k-1} \leq v^{1/3^{l-2}}}}  \dots  \sum_{\substack{p_{k-l+3} \in \p_{p_{k-l+4}}\\ p_{k-l+3} \leq v^{1/3}}}\bigg(\frac{\pi(v^{1/3};p_{k-l+3},1)}{v^{1/3}}\\ & \hskip25mm+\int_{2}^{v^{1/3}} \frac{\pi(u;p_{k-l+3},1)}{u^{2}}du \bigg)+O\bigg(\frac{v\log y}{\log v}+\li(v)(\log\log v)^{l-2} \bigg).
\end{align*}
Inside the sum by trivially estimating $\pi(x;q,1)$ by $x/q$ inside the sum and using Brun--Titchmarsh \eqref{BT} we get
\begin{align*}
\sum_{q \leq y^{k}}& \log q \sum_{a \in \nat} \sum_{\substack{p_{k} \in \p_{q^{a}} \\ p_{k} \leq v^{1/3^{l-1}}}} \sum_{\substack{p_{k-1} \in \p_{p_{k}}\\ p_{k-1} \leq v^{1/3^{l-2}}}}  \dots  \sum_{\substack{p_{k-l+3} \in \p_{p_{k-l+4}}\\ p_{k-l+3} \leq v^{1/3}}}\frac{\pi(v^{1/3};p_{k-l+3},1)}{v^{1/3}} \\
& \ll \sum_{q \leq y^{k}} \log q \sum_{a \in \nat} \sum_{\substack{p_{k} \in \p_{q^{a}} \\ p_{k} \leq v^{1/3^{l-1}}}} \sum_{\substack{p_{k-1} \in \p_{p_{k}}\\ p_{k-1} \leq v^{1/3^{l-2}}}}  \dots  \sum_{\substack{p_{k-l+3} \in \p_{p_{k-l+4}}\\ p_{k-l+3} \leq v^{1/3}}}\frac{1}{p_{k-l+3}} \\
& \ll \sum_{q \leq y^{k}} \log q \sum_{a \in \nat} \frac{(\log\log v)^{l-2}}{q^a}\\
& \ll \sum_{q \leq y^{k}} \log q \frac{(\log\log v)^{l-2}}{q}\\
& \ll (\log\log v)^{l-2}\log y.
\end{align*} Multiplying through by $\li(v)$ finishes the lemma.
\end{proof}

We now require a lemma to find the asymptotic formula for $h_k$ using the previous recurrence relation
\begin{lemma}\label{hk}
Let $2\le l\le k$.
$$g_{k,l}(u)=\frac{ku(\log\log u)^{l-1}\log y}{(l-1)!\log u}+O\bigg(\frac{u(\log\log u)^{l-1}}{\log u}+\frac{u(\log\log u)^{l-2}\log ^2y}{\log u}\bigg)$$
which implies
\begin{align*}
\sum_{p \leq t}\hk (p) & = \frac{kt(\log\log t)^{k-1}\log y}{(k-1)!\log t}+O\bigg(\frac{t(\log\log t)^{k-1}}{\log t}\\ & \hskip10mm+\frac{t(\log\log t)^{k-2}\log ^2y}{\log t}+t^{1-1/3^k}\log t(\log\log t)^{k-2}y^k\bigg).
\end{align*}
\end{lemma}
\begin{proof}
The second formula is derived from the first by setting $l=k$, $u=t$ and using Lemma \ref{hk1}. We'll proceed with the first formula by induction on $l$. Using the estimates we obtained via Bombieri--Vinogradov in Lemma \ref{BVreduction},  we have for $l=2$

\begin{align*}g_{k,2}(u) &= \sum_{q \leq y^{k}} \log q \sum_{a \in \nat} \sum_{\substack{p_{k} \in \p_{q^{a}} \\ p_{k} \leq u^{1/3}}}\pi(u;p_k,1) \\
& = \li(u)\sum_{q \leq y^{k}} \log q \sum_{a \in \nat} \sum_{\substack{p_{k} \in \p_{q^{a}} \\ p_{k} \leq u^{1/3}}} \frac{1}{p_{k}} + O\bigg(\li(u)+\frac{u \log y}{\log u}\bigg).\end{align*}
We then use \eqref{BT3} and 
$$\log\log(u^{1/3})=\log\log u + O(1)$$
to get
\begin{align*}g_{k,2}(u)&=\li(u)\sum_{q \leq y^{k}} \log q \sum_{a \in \nat}\bigg( \frac{\log \log u^{1/3}}{\phi(q^{a})}+O\bigg(\frac{\log(q^{a})}{\phi(q^{a})} \bigg)\bigg)+ O\bigg(\frac{u \log y}{\log u}\bigg) \\
& =\li(u)(\log \log u + O(1))\sum_{q \leq y^{k}} \log q \sum_{a \in \nat}\bigg( \frac{1}{q^{a}}+O\bigg(\frac{1}{q^{a+1}}\bigg)\bigg)+O\bigg(\li(u)\sum_{q \leq y^{k}} \log^{2} q \sum_{a \in \nat} \frac{a}{q^{a}}\bigg) \\
& \hskip10mm+ O\bigg(\frac{u \log y}{\log u}\bigg) \\
&=\li(u)(\log \log u + O(1))\sum_{q \leq y^{k}} \bigg(\frac{\log q}{q} + O\bigg(\frac{\log q}{q^{2}}\bigg)\bigg)+O\bigg(\li(u)\sum_{q \leq y^{k}} \frac{\log^{2} q}{q}+\frac{u \log y}{\log u}\bigg)\\
&=\li(u)\log \log u \log(y^k) + O\bigg(\li(u)(\log y + \log \log u + \log^{2}y) +\frac{u \log y}{\log u}\bigg) \\
& =\frac{ku \log \log u \log y}{\log u} + O\bigg(\frac{u \log \log u}{\log u} +\frac{u \log^{2} y}{\log u}\bigg),\end{align*} completing the base case.
Now using Lemma \ref{Recursion} we get
\begin{align*}g_{k,l}(v)& =\li(v)\int_{2}^{v^{1/3}} \frac{1}{u^{2}} g_{k,l-1}(u)du + O\bigg(\frac{v (\log \log v)^{l-2}\log y}{\log v} \bigg)\\
&=\li(v)\int_{2}^{v^{1/3}} \frac{1}{u^{2}}\bigg(\frac{ku(\log\log u)^{l-2}\log y}{(l-2)!\log u}+O\bigg(\frac{u(\log\log u)^{l-2}}{\log u}+\\
& \hskip10mm\frac{u(\log\log u)^{l-3}\log ^2y}{\log u}\bigg)\bigg)du + O\bigg(\frac{v (\log \log v)^{l-2}\log y}{\log v} \bigg) \\
&=\li(v)\int_{2}^{v^{1/3}} \bigg(\frac{k(\log\log u)^{l-2}\log y}{(l-2)!u\log u}+O\bigg(\frac{(\log\log u)^{l-2}}{u\log u}+\frac{(\log\log u)^{l-3}\log ^2y}{u\log u}\bigg)\bigg)du\\
& \hskip10mm + O\bigg(\frac{v (\log \log v)^{l-2}\log y}{\log v} \bigg)\\
&= \frac{k\li(v)(\log\log v^{1/3})^{l-1}\log y}{(l-1)!}+O\bigg(\li(v)(\log\log v^{1/3})^{l-1}+  \li(v)(\log\log v^{1/3})^{l-2}\log ^2y \\
& \hskip10mm +\frac{v (\log \log v)^{l-2}\log y}{\log v} \bigg).
\end{align*}
Once again by using $$\log\log v^{1/3}=\log\log v +O(1)$$ we get
\begin{align*}
& \frac{kv(\log\log v)^{l-1}\log y}{(l-1)!\log v} +O\bigg(\frac{v(\log\log v)^{l-1}}{\log v}+
\frac{v(\log\log v)^{l-2}\log ^2y}{\log v}+\frac{v (\log \log v)^{l-2}\log y}{\log v} \bigg) \\
&=\frac{kv(\log\log v)^{l-1}\log y}{(l-1)!\log v} +O\bigg(\frac{v(\log\log v)^{l-1}}{\log v}+
\frac{v(\log\log v)^{l-2}\log ^2y}{\log v}\bigg),
\end{align*} completing the induction.
\end{proof}

\section{The Proof of the First Moment}
We now are in a position to prove the propostion for the first moment.
\begin{proof}[Proof of Proposition \ref{M1}]
\begin{align*} M_{1}(x)& = \sum_{p \leq x}\frac{\hk (p)}{p} \\
& = \sum_{p \leq e^{e}}\frac{\hk (p)}{p} + \sum_{e^{e} < p \leq x}\frac{\hk (p)}{p} \\
& = O(1) + \sum_{e^{e} < p \leq x}\hk (p)\bigg(\frac{1}{x}+ \int_{p}^{x}\frac{dt}{t^{2}}\bigg) \\
& =O(1)+ \frac{1}{x}\sum_{e^{e} < p \leq x}\hk (p) + \int_{e^{e}}^{x} \frac{dt}{t^{2}}\sum_{e^{e} < p \leq t}\hk (p).\end{align*}
Using $t=x$ in Lemma \ref{hk} we get that 
$$\sum_{e^{e} < p \leq x}\hk (p) \ll \frac{xy^{k-1}\log y}{\log x}$$ and since $$\sum_{e^{e} < p \leq t}\hk (p)$$ differs from $$\sum_{p \leq t}\hk (p)$$ by a constant, we get that

\begin{align*}M_{1}(x)&= O(1) + \frac{1}{x}O\bigg(\frac{xy^{k-1}\log y}{\log x}\bigg) +  \int_{e^{e}}^{x} \frac{dt}{t^{2}}\bigg(\frac{kt(\log\log t)^{k-1}\log y}{(k-1)!\log t}+O\bigg(\frac{t(\log\log t)^{k-1}}{\log t}\\ & \hskip10mm+\frac{t(\log\log t)^{k-2}\log ^2y}{\log t}+t^{1-1/3^k}\log t(\log\log t)^{k-2}y^k\bigg)\bigg)\end{align*} using Lemma \ref{hk}. Noting that 
\begin{align*}\int_{e^{e}}^{x}& \frac{dt}{t^{2}}t^{1-1/3^k}\log t(\log\log t)^{k-2}y^k \\
& = \int_{e^{e}}^{x}\frac{y^k dt}{t^{1+\epsilon}}\\
& \ll y^k \end{align*} yields
\begin{align*}
O(y^k)&+O\bigg(\frac{y^{k-1}\log y}{\log x}\bigg) +  \int_{e^{e}}^{x} \frac{dt}{t^{2}}\bigg(\frac{kt(\log\log t)^{k-1}\log y}{(k-1)!\log t}+O\bigg(\frac{t(\log\log t)^{k-1}}{\log t}\\ & \hskip10mm+\frac{t(\log\log t)^{k-2}\log ^2y}{\log t}\bigg)\bigg) \\
& = O(y^k)+  \frac{k(\log\log x)^{k}\log y}{k(k-1)!}+O\bigg((\log\log x)^{k}+ (\log\log x)^{k-1}\log ^2y\bigg) \\
& = \frac{y^{k}\log y}{(k-1)!}+O(y^k) \end{align*} as needed.
\end{proof}

\section{The Proof of the Second Moment}
We now turn our attention to the second moment. Our first lemma will bound the case where $p_3=r_3$ and then we'll use the summations from Lemma \ref{Doubles} to take care of the rest.

\begin{lemma} \label{DoubleSum}
\begin{align*}\sum_{q_{1},q_{2} \le y^{k}}\log q_{1} \log q_{2}& \sum_{a_{1},a_{2} \in \nat}\sum_{\substack{p_{k} \in \p_{q_{1}^{a_{1}}} \\ r_{k} \in \p_{q_{2}^{a_{2}}}}}\sum_{\substack{p_{k-1} \in \p_{p_{k}} \\ r_{k-1} \in \p_{r_{k}}}}\dots \sum_{\substack{p_{3} \in \p_{p_{4}} \\ r_{3} \in \p_{r_{4}}}}\sum_{s \in \p_{p_{3}} \cap \p_{r_{3}}}\sum_{\substack{p \le t \\ p \in \p_{s}}}1 \\ & \ll t^{1-\eps}y^{k}\log y + \frac{t(\log \log t)^{2k-2}}{\log t}\log ^{2} y \end{align*} for some $\epsilon > 0.$
\end{lemma}
\begin{proof}
Our sum is 
\begin{align*}\sum_{q_{1},q_{2} \le y^{k}}&\log q_{1} \log q_{2}\sum_{a_{1},a_{2} \in \nat}\sum_{\substack{p_{k} \in \p_{q_{1}^{a_{1}}} \\ r_{k} \in \p_{q_{2}^{a_{2}}}}}\sum_{\substack{p_{k-1} \in \p_{p_{k}} \\ r_{k-1} \in \p_{r_{k}}}}\dots \sum_{\substack{p_{3} \in \p_{p_{4}} \\ r_{3} \in \p_{r_{4}}}}\sum_{s \in \p_{p_{3}} \cap \p_{r_{3}}}\sum_{\substack{p \le t \\ p \in \p_{s}}}1 \\
& =\sum_{q_{1},q_{2} \le y^{k}}\log q_{1} \log q_{2}\sum_{a_{1},a_{2} \in \nat}\sum_{\substack{p_{k} \in \p_{q_{1}^{a_{1}}} \\ r_{k} \in \p_{q_{2}^{a_{2}}}}}\sum_{\substack{p_{k-1} \in \p_{p_{k}} \\ r_{k-1} \in \p_{r_{k}}}}\dots \sum_{\substack{p_{3} \in \p_{p_{4}} \\ r_{3} \in \p_{r_{4}}}}\sum_{s \in \p_{p_{3}r_3}}\pi(t;s,1).\end{align*}
We split up into two cases. If $q_{1}^{a_{1}}q_{2}^{a_{2}} > t^{\alpha}$, then suppose $q_{1}^{a_{1}}>t^{\alpha/2}.$ (the other case is analogous) We get from the trivial bound on $\pi(t;s,1)$ that
\begin{align*}\sum_{q_{1},q_{2} \le y^{k}}& \log q_{1} \log q_{2}\sum_{\substack{a_{1},a_{2} \in \nat \\ q_{1}^{a_{1}}>t^{\frac{\alpha}{2}}}}\sum_{\substack{p_{k} \in \p_{q_{1}^{a_{1}}} \\ r_{k} \in \p_{q_{2}^{a_{2}}}}}\sum_{\substack{p_{k-1} \in \p_{p_{k}} \\ r_{k-1} \in \p_{r_{k}}}}\dots \sum_{\substack{p_{3} \in \p_{p_{4}} \\ r_{3} \in \p_{r_{4}}}}\sum_{s \in \p_{p_{3}r_3}}\pi(t;s,1) \\
& = \sum_{q_{1},q_{2} \le y^{k}}\log q_{1} \log q_{2}\sum_{\substack{a_{1},a_{2} \in \nat \\ q_{1}^{a_{1}}>t^{\frac{\alpha}{2}}}}\sum_{\substack{p_{k} \in \p_{q_{1}^{a_{1}}} \\ r_{k} \in \p_{q_{2}^{a_{2}}}}}\sum_{\substack{p_{k-1} \in \p_{p_{k}} \\ r_{k-1} \in \p_{r_{k}}}}\dots \sum_{\substack{p_{3} \in \p_{p_{4}} \\ r_{3} \in \p_{r_{4}}}} \sum_{s \in \p_{p_{3}r_3}}\frac{t\log t}{s} \\
& = \sum_{q_{1},q_{2} \le y^{k}}\log q_{1} \log q_{2}\sum_{\substack{a_{1},a_{2} \in \nat \\ q_{1}^{a_{1}}>t^{\frac{\alpha}{2}}}}\sum_{\substack{p_{k} \in \p_{q_{1}^{a_{1}}} \\ r_{k} \in \p_{q_{2}^{a_{2}}}}}\sum_{\substack{p_{k-1} \in \p_{p_{k}} \\ r_{k-1} \in \p_{r_{k}}}}\dots \sum_{\substack{p_{3} \in \p_{p_{4}} \\ r_{3} \in \p_{r_{4}}}}\frac{t\log t \log\log t}{p_3r_3} \\
& = \sum_{q_{1},q_{2} \le y^{k}}\log q_{1} \log q_{2}\sum_{\substack{a_{1},a_{2} \in \nat \\ q_{1}^{a_{1}}>t^{\frac{\alpha}{2}}}}\frac{t\log t (\log\log t)^{2k-3}}{q_1^{\alpha_1}q_2^{\alpha_2}}.
\end{align*}
By letting $A = \min \{a | q_{1}^{a_{1}}>t^{\frac{\alpha}{2}}\} $ we get 

\begin{align*}& \ll \sum_{q_{1},q_{2} \le y^{k}}\log q_{1} \log q_{2}\frac{t \log t (\log \log t)^{k-1}}{q_{1}^{A}q_{2}} \\
& \le t^{1-\frac{\alpha}{2}} \log t (\log \log t)^{2k-3}\sum_{q_{1} \le y^{k}}\log q_{1}\sum_{q_{2} \le y^{k}}\frac{\log q_{2}}{q} \\ & \ll t^{1-\eps}y^{k} \log y.\end{align*}
If $q_{1}^{a_{1}}q_{2}^{a_{2}} > t^{\alpha}$, then by Lemma \ref{Doubles} part (d) we get 
\begin{align*}\sum_{q_{1},q_{2} \le y^{k}}& \log q_{1} \log q_{2}\sum_{\substack{a_{1},a_{2} \in \nat \\ q_{1}^{a_{1}}q_{2}^{a_{2}}\le t^{\alpha}}}\sum_{\substack{p_{k} \in \p_{q_{1}^{a_{1}}} \\ r_{k} \in \p_{q_{2}^{a_{2}}}}}\sum_{\substack{p_{k-1} \in \p_{p_{k}} \\ r_{k-1} \in \p_{r_{k}}}}\dots \sum_{\substack{p_{3} \in \p_{p_{4}} \\ r_{3} \in \p_{r_{4}}}}\sum_{s \in \p_{p_{3}r_3}}\pi(t;s,1) \\
& \ll \sum_{q_{1},q_{2} \le y^{k}}\log q_{1} \log q_{2}\sum_{\substack{a_{1},a_{2} \in \nat \\ q_{1}^{a_{1}}q_{2}^{a_{2}}\le t^{\alpha}}}\frac{t (\log\log t)^{2k-2}}{q_{1}^{a_{1}}q_{2}^{a_{2}} \log t} \\
& \ll \sum_{q_{1},q_{2} \le y^{k}}\log q_{1} \log q_{2}\frac{t (\log\log t)^{2k-2}}{q_{1}q_{2} \log t} \\
& = \frac{t (\log\log t)^{2k-2}}{\log t}\bigg( \sum_{q \le y^{k}}\frac{\log q}{q}\bigg)^2 \\
& \ll \frac{t (\log\log t)^{2k-2}}{\log t}\log^2 y \end{align*} by \eqref{Merten}, completing the lemma.

\end{proof}

We now have enough to finish the second moment which is the final piece of the puzzle.

\begin{proof}[Proof of Proposition~\ref{M2}]
\begin{align*}\sum_{p \leq t}\hk (p)^{2} &=\sum_{p \leq x}\bigg(\sum_{p_{1}\mid p} \sum_{p_{2}\mid p_{1}-1} \dots \sum_{p_{k}\mid p_{k-1}-1} \sum_{q \leq y^{k}} \nu_{q}(p_{k}-1)\log q\bigg)^{2}
\\ &=\sum_{q_{1},q_{2} \le y^{k}}\log q_{1} \log q_{2}\sum_{a_{1},a_{2} \in \nat}\sum_{\substack{p_{k} \in \p_{q_{1}^{a_{1}}} \\ r_{k} \in \p_{q_{2}^{a_{2}}}}}\sum_{\substack{p_{k-1} \in \p_{p_{k}} \\ r_{k-1} \in \p_{r_{k}}}}\dots \sum_{\substack{p_{2} \in \p_{p_{3}} \\ r_{2} \in \p_{r_{3}}}}\sum_{\substack{p \le t \\ p \in \p_{p_{2}} \\ p \in \p_{r_{2}}}}1
\end{align*}
since the condition $p_{1} \mid p$ only occurs if $p_{1} = p.$ We then split up the sum according to whether or not $p_2=r_2.$ Lemma \ref{DoubleSum} deals with the part where $s=p_2=r_2$ leaving us with

\begin{align*}\sum_{q_{1},q_{2} \le y^{k}}\log q_{1} \log q_{2} &\sum_{a_{1},a_{2} \in \nat}\sum_{\substack{p_{k} \in \p_{q_{1}^{a_{1}}} \\ r_{k} \in \p_{q_{2}^{a_{2}}}}}\sum_{\substack{p_{k-1} \in \p_{p_{k}} \\ r_{k-1} \in \p_{r_{k}}}}...\sum_{\substack{p_{2} \in \p_{p_{3}} \\ r_{2} \in \p_{r_{3}} \\ p_{2} \ne r_{2}}}\sum_{\substack{p \le t \\ p \in \p_{p_{2}} \\ p \in \p_{r_{2}}}}1 \\ &+ O\bigg(t^{1-\eps}y^{k}\log y + \frac{t(\log \log t)^{2k-2}}{\log t}\log ^{2} y  \bigg).\end{align*}
The sum becomes 
$$\sum_{q_{1},q_{2} \le y^{k}}\log q_{1} \log q_{2} \sum_{a_{1},a_{2} \in \nat}\sum_{\substack{p_{k} \in \p_{q_{1}^{a_{1}}} \\ r_{k} \in \p_{q_{2}^{a_{2}}}}}\sum_{\substack{p_{k-1} \in \p_{p_{k}} \\ r_{k-1} \in \p_{r_{k}}}}\dots\sum_{\substack{p_{2} \in \p_{p_{3}} \\ r_{2} \in \p_{r_{3}}}} \pi(t;p_{2}r_{2},1).$$
If $q_{1}^{a_{1}} > t^{\alpha_{1}}$, then so is $p_{2}$, and hence by \eqref{Doublea} we get 

\begin{align*}&\sum_{q_{1},q_{2} \le y^{k}}\log q_{1} \log q_{2} \sum_{\substack{a_{1},a_{2} \in \nat \\ q_{1}^{a_{1}} > t^{\alpha_{1}}}} \sum_{\substack{p_{k} \in \p_{q_{1}^{a_{1}}}\\ r_{k} \in \p_{q_{2}^{a_{2}}}}}\sum_{\substack{p_{k-1} \in \p_{p_{k}} \\ r_{k-1} \in \p_{r_{k}}}}\dots\sum_{\substack{p_{3} \in \p_{p_{4}} \\ r_{3} \in \p_{r_{4}}}} \frac{t\log^{2}t}{p_{3}r_{3}}
\\ &\ll \sum_{q_{1},q_{2} \le y^{k}}\log q_{1} \log q_{2} \sum_{\substack{a_{1},a_{2} \in \nat \\ q_{1}^{a_{1}} > t^{\alpha_{1}}}}\frac{t\log^{2}t (\log\log t)^{2k-4}}{q_{1}^{a_{1}}q_{2}^{a_{2}}}
\\ &\ll  t^{1-\alpha_{1}}\log^{2}t (\log\log t)^{2k-4}\sum_{q_{1},q_{2} \le y^{k}}\log q_{1} \log q_{2} \sum_{a_{2} \in \nat }\frac{1}{q_{2}^{a_{2}}}
\\ &\ll  t^{1-\alpha_{1}}\log^{2}t (\log\log t)^{2k-4}\sum_{q_{1},q_{2} \le y^{k}}\frac{\log q_{1} \log q_{2}}{q_{2}}
\\ &\ll  t^{1-\alpha_{1}}\log^{2}t (\log\log t)^{2k-4}  (y^{k}\log y ). \end{align*}
We similarly get the same bound if $q_{2}^{a_{2}} > t^{\alpha_{1}}.$ If neither of $q_{1}^{a_{1}},q_{2}^{a_{2}}$ exceed $ t^{\alpha_{1}},$ then by \eqref{Doublec} and using that for $b_{i}=q_{i}^{a_{i}}$
$$\frac{b_{i}}{\phi(b_{i})} \ll 1, \frac{1}{\phi(b_{i})} \ll \frac{1}{b_{i}},$$ we get
\begin{align*}\sum_{q_{1},q_{2} \le y^{k}}\log q_{1} &\log q_{2} \sum_{\substack{a_{1},a_{2} \in \nat  \\ q_{1}^{a_{1}},q_{2}^{a_{2}} \le  t^{\alpha_{1}}}}\sum_{\substack{p_{k} \in \p_{q_{1}^{a_{1}}} \\ r_{k} \in \p_{q_{2}^{a_{2}}}}}\sum_{\substack{p_{k-1} \in \p_{p_{k}} \\ r_{k-1} \in \p_{r_{k}}}}\dots\sum_{\substack{p_{i} \in \p_{p_{i+1}} \\ r_{i} \in \p_{r_{i+1}}}} \sum_{\substack{p_{i-1} \in \p_{p_{i}} \\ r_{i-1} \in \p_{r_{i}}}} \dots \sum_{\substack{p_{2} \in \p_{p_{3}} \\ r_{2} \in \p_{r_{3}}}} \pi(t;p_{2}r_{2},1) \\
& \ll \sum_{q_{1},q_{2} \le y^{k}}\log q_{1} \log q_{2} \sum_{\substack{a_{1},a_{2} \in \nat  \\ q_{1}^{a_{1}},q_{2}^{a_{2}} \le  t^{\alpha_{1}}}}\frac{t(\log \log t)^{2k-2}}{q_{1}^{a_{1}}q_{2}^{a_{2}}\log t} \\
& \ll \frac{t(\log \log t)^{2k-2}}{\log t} \sum_{q_{1},q_{2} \le y^{k}}\frac{\log q_{1} \log q_{2}}{q_{1}q_{2}} \\
& \ll \frac{t(\log \log t)^{2k-2}\log ^2 y}{\log t}.
\end{align*}
Hence the above gives us that
$$\sum_{p \leq t}\hk (p)^{2} \ll t^{1-\eps}y^{k}\log y + \frac{t(\log \log t)^{2k-2}\log ^{2} y}{\log t}.$$
Using partial summation we have 
\begin{align*}
M_2(x) &= \sum_{p \leq x}\frac{\hk (p)^{2}}{p} =  \sum_{p \leq e^{e}}\frac{\hk (p)^{2}}{p} + \frac{1}{x}\sum_{e^{e} \le p \leq x}\hk (p)^{2} + \int_{e^e}^{x}\frac{dt}{t^{2}}\sum_{e^{e} \le p \leq t}\hk (p)^{2} \\
& \ll 1+\frac{1}{x}\bigg( x^{1-\eps}y^{k}\log y + \frac{x(\log \log x)^{2k-2}\log ^{2} y}{\log x}\bigg) \\
& \hskip10mm+ \int_{e^e}^{x}\bigg(t^{-1-\eps}y^{k}\log y + \frac{(\log \log t)^{2k-2}\log ^{2} y}{t\log t}\bigg)dt \\
&\ll \frac{y^{2k-2}\log ^{2} y}{\log x} + x^{-\eps}y^k \log y + (\log \log x)^{2k-1}\log ^2 y \\
&\ll y^{2k-1}\log ^2 y
\end{align*} completing the proof of Proposition \ref{M2} and hence Theorem \ref{MainTheorem}.
\end{proof}

\section{Theorem \ref{Second}}\label{Theorem2}

We now turn our attention to the proof of Theorem \ref{Second}. It will be necessary to use the following upper bound for the Carmichael function of a product.

\begin{lemma}
Let $a,b$ be natural numbers, then
\begin{equation}\label{lambdabound}
\lambda(ab) \le b\lambda(a).
\end{equation} 
\begin{proof}
We first note that it suffices to show the inequality whenever $b$ is prime, because if 
$$b=p_1\dots p_k$$ where the $p_i$ are not necessarily distinct, then repeated use of the theorem where $b$ is prime yields

$$\lambda(ab) = \lambda(ap_1\dots p_k) \le p_1\lambda(ap_2\dots p_k)\le \dots \le p_1\dots p_k\lambda(a) = b\lambda(a).$$
If $b$ is a prime which divides $a$, then 
$$a=b^ep_1^{e_1}\dots p_k^{e_k} \text{ and } ab=b^{e+1}p_1^{e_1}\dots p_k^{e_k}.$$
Therefore \begin{align*}\lambda(ab)&=\lcm\bigg(\lambda(b^{e+1}),\lambda(p_1^{e_1}),\dots ,\lambda(p_k^{e_k})\bigg) \\ 
&\le \lcm\bigg(b\lambda(b^{e}),\lambda(p_1^{e_1}),\dots,\lambda(p_k^{e_k})\bigg) \\
& \le b * \lcm\bigg(\lambda(b^{e}),\lambda(p_1^{e_1}),\dots,\lambda(p_k^{e_k})\bigg) \\ 
& = b\lambda(a)\end{align*} where the first inequality is in fact an equality if $b^e=4$. Also note that in this case, it would not be hard to show that $\lambda(ab) \mid b\lambda(a).$ If $(a,b)=1$, then

\begin{align*}\lambda(ab)&=lcm\bigg(b-1,\lambda(p_1^{e_1}),\dots,\lambda(p_k^{e_k})\bigg) \\ 
& \le (b-1) \lcm\bigg(\lambda(p_1^{e_1}),\dots,\lambda(p_k^{e_k})\bigg) \\ 
& <  b\lambda(a),\end{align*} ending the proposition.
\end{proof}
\end{lemma}

Suppose that $g(n)$ is an arithmetic function of the form $\phi(h(n))$ where $h(n)$ is a $(k-1)$--fold iterate involving $\phi$ and $\lambda$. Then we can use equation \eqref{lambdabound} to get 

$$\lambda_{l+k}(n) \le \lambda_l(g(n)) \le \lambda_l\bigg(\frac{g(n)}{\lambda_k(n)}\lambda_k(n)  \bigg) \le \lambda_{l+k}(n)\frac{g(n)}{\lambda_k(n)}.$$
Since $g(n)\le n$ we have that $$\frac{g(n)}{\lambda_k(n)} \le \frac{n}{\lambda_k(n)} = \exp\bigg(\frac{1}{(k-1)!}(1+o_k(1)) (\log\log{n})^{k}\log\log\log{n} \bigg)$$ by Theorem \ref{MainTheorem} and hence

$$\lambda_{l+k}(n) \le \lambda_l(g(n)) \le \lambda_l\bigg(\frac{g(n)}{\lambda_k(n)}\lambda_k(n)  \bigg) \le \lambda_{l+k}(n)\exp\bigg(\frac{1}{(k-1)!} (\log\log{n})^{k}(1+o_k(1))\log\log\log{n} \bigg).$$
From the fact that 
$$\lambda_{l+k}(n)=n\exp\bigg(-\frac{1}{(k+l-1)!}(1+o_{l,k}(1))(\log\log{n})^{k+l}\log\log\log{n} \bigg)$$ we get 

$$\lambda_l(g(n))=n\exp\bigg(-\frac{1}{(k+l-1)!}(1+o_{l,k}(1))(\log\log{n})^{k+l}\log\log\log{n} \bigg).$$
As for $\phi(g(n))$ we note that unless $g(n)=\phi_k(n)$, $g(n)$ can be writen as $\phi_l(h(n))$ where $h(n)$ is a $(k-l)$--fold iterate beginning with a $\lambda$. From above we can see that

$$h(n)=n\exp\bigg(-\frac{1}{(k-l-1)!}(1+o_{k}(1))(\log\log{n})^{k-l}\log\log\log{n} \bigg)$$ and so $\phi(h(n))$ is bounded above by $h(n)$ and below by 
\begin{align*}\frac{h(n)}{e^{\gamma}\log\log h(n) + \frac{3}{\log\log h(n)}} &= \frac{h(n)}{e^{\gamma}\log\big(\log n -  \frac{1}{(k-l-1)!}(1+o_{k}(1))(\log\log{n})^{k-l}\log\log\log{n} \big)} \\
& = \frac{h(n)}{e^{\gamma}\log\log n -  O\big(\frac{1}{(k-l-1)!\log n}(1+o_{k}(1))(\log\log{n})^{k-l}\log\log\log{n} \big)} \\
& =h(n)\exp\big(O(\log\log\log n) \big)
\end{align*} which is within the error of $h(n)$. Hence any string of $\phi$ will not change our estimate. Therefore if $g(n)$ is a $k$--fold iteration of $\phi$ and $\lambda$ which is not $\phi_k(n)$, but which begins with $l$ copies of $\phi$, then 

$$g(n)=n\exp\bigg(-\frac{1}{(k-l-1)!}(1+o_{k}(1))(\log\log{n})^{k-l}\log\log\log{n} \bigg)$$ yielding our theorem.

\section*{Acknowledgements}
The author would like to thank Greg Martin for his assistance and guidance.

\end{document}